\documentclass[11pt,a4paper]{article}
\usepackage{fullpage}
\usepackage{authblk}

\usepackage{mathrsfs}
\usepackage{amsfonts}
\usepackage{amssymb}
\usepackage{bm}
\usepackage{multirow,bigstrut}
\usepackage{threeparttable}
\usepackage{enumerate}
\usepackage{tikz}
\usepackage{graphicx}
\usepackage{stfloats}
\usepackage{array}
\usepackage{amsmath}
\usepackage[colorlinks,
            linkcolor=red,
            anchorcolor=blue,
            citecolor=green
            ]{hyperref}

\usepackage{algorithm}
\usepackage{algorithmic}
\usepackage{setspace}

\newenvironment{proof}{\noindent{\em \textbf{Proof.}}}{\quad \hfill$\Box$\vspace{2ex}}

\renewcommand{\Re}{\mathcal{R}}

\newtheorem{theorem}{Theorem}[section]
\newtheorem{definition}[theorem]{Definition}
\newtheorem{remark}[theorem]{Remark}
\newtheorem{corollary}[theorem]{Corollary}
\newtheorem{lemma}[theorem]{Lemma}
\newtheorem{example}[theorem]{Example}
\newtheorem{proposition}[theorem]{\bf Proposition}
\numberwithin{equation}{section}

\def \H {\mathbb{H}}
\def \T {\mathbb{T}}
\def \C {\mathbb{C}}
\def \rc {\mathcal{R}}
\def\qi { \bm {i}}
\def\qj { \bm{j}}
\def\qk { \bm{k}}
\def\sgn{\mbox{sgn}}
\def\Arg{\mbox{Arg}}
\def\Ln{\mathrm{Ln}}
\def \cx {\mathfrak{C}}
\def \bphi {\bm{\phi}}
\def \bvphi {\bm{\varphi}}
\def \bfunc {\bm{f}}
\def \bq {\bm{q}}
\def \bmeta {\bm{\eta}}

\def \bx {\bm{x}}
\def \by {\bm{y}}

\def \bme {\bm{e}}
\def \qd {\widehat{\det}}

\newcommand{\abs}[1]{\left|#1\right|}

\title{A Unified Analysis of Linear Quaternion Dynamic Equations on Time Scales}
	\author[a]{Dong Cheng\thanks{chengdong720@163.com}}
	
	\author[a]{Kit Ian Kou\thanks{kikou@umac.mo}}
	
	\author[c]{Yong Hui Xia \thanks{xiadoc@163.com}}
	
	\affil[a]{\normalsize{Department of Mathematics, Faculty of Science and Technology, University of Macau, Macao, China}}
	\affil[c]{\normalsize{Department of Mathematics, Zhejiang Normal University, Jinhua, China}}
	
	\date{}

\begin{document}
  \maketitle
\begin{abstract}
	\normalsize

	Over the last years, considerable attention has been paid to the role of the quaternion differential equations (QDEs)
which extend the ordinary differential equations. The theory of QDEs was recently well established and it has wide applications in physics and life science. This paper establishes a systematic frame work for the theory of                                                                                                                                                                                                                                                                                                                                                                                                                                                                                                                                                                                                    linear quaternion dynamic equations on time scales (QDETS), which can be applied to wave phenomena modeling, fluid dynamics and filter design. The algebraic structure of the solutions to the QDETS is actually a left- or right- module, not a linear vector space. On the non-commutativity of the quaternion algebra, many concepts and properties of the classical dynamic equations on time scales (DETS) can not be applied. They should be redefined accordingly. Using $q$-determinant, a novel definition of Wronskian is introduced under the framework of quaternions which is different from the
standard one in DETS. Liouville's formula
for QDETS is also analyzed. Upon these, the
solutions to the linear QDETS are established. The Putzer's algorithms
to evaluate the fundamental solution matrix for homogeneous QDETS are presented. Furthermore, the
variation of constants formula to solve the nonhomogeneous QDETs is given.
Some concrete examples are provided
to illustrate the feasibility of the proposed algorithms.
\end{abstract}

 \begin{keywords}
		Dynamic systems on time scales, difference  equations, fundamental solution matrix, quaternions.
\end{keywords}

\begin{msc}
		34N05, 34A30, 39A06, 20G20.
\end{msc}

\section{Introduction}\label{S1}

The theory of dynamic equations on time scales (DETS) has enormous applications \cite{atici2006application,marks2008generalized}.  It is applicable to many fields in which physical phenomena can be described by  continuous or discrete dynamical models. For instance, both continuous and discrete models are used in 3D tracking of shape, motion estimation \cite{metaxas2005dynamic} and DNA dynamics \cite{klapper1998remarks}. An unify framework for the theory of DETS was introduced in 1988 by Hilger \cite{hilger1988thesis, hilger1990analysis}. It  unifies continuous and discrete dynamic equations. The classical differential and difference equations are special cases of DETS. 

Over the years, the theory of quaternion  differential equations (QDEs) has received a lot of attention \cite{wilczynski2009quaternionic,wilczynski2012quaternionic,campos2006periodic,gasull2009one,zhang2011global}. The QDEs have numerous applications in physics and engineering, such as spatial kinematic modelling and attitude dynamics \cite{chou1992quaternion,gupta1998linear}, fluid mechanics \cite{gibbon2002quaternionic,gibbon2006quaternions}, quantum mechanics \cite{alder1986quaternionic,adler1995quaternionic} and so on. Recently, the basic theory and fundamental results of linear QDEs  was established \cite{kou2015linear1,kou2015linear2,xia2016algorithm}.  It is  interesting and necessary to  extend the  theory of QDEs to quaternion dynamic equations on arbitrary time scales (QDETS), so that the theory  of quaternion dynamic equations can be widely applied to physical and engineering problems. On  the one hand, both discrete and mixture of continuous and discrete dynamical models are subsumed within the QDETS. On the other hand, some differential equations need to be integrated into difference equations for computations or simulations.  For example,  the kinematics system on discrete time was studied in \cite{sola2012quaternion}.

The main purpose  of this paper is to study  the basic theory of linear QDETS. The Hilger quaternion numbers on time scales was studied in \cite{georgiev2013introduction}. The researchers further gave the definition of the quaternion exponential function on time scales \cite{georgiev2013introduction}. However, we will show that the quaternion exponential function, in general, is not the solution of one-dimensional homogeneous linear QDETS:
\begin{equation*}
	y^{\Delta} =p(t)y,
\end{equation*}
unless $p(t)$ is  either  a real-valued function or  a quaternion constant function.  This is a striking  difference between DETS and QDETS.  Owing  to the non-commutativity of multiplication of quaternions, there are many concepts   of DETS are not effective for QDETS. Besides, in consideration of the differences  between QDETS and QDEs in nature, lots of results concerning QDEs can not be easily  carried  out to the corresponding results of QDETS. The product rule of delta derivative  on time scales is more tedious than  traditional derivative. Therefore  the Wronskian defined in   \cite{kou2015linear2}  is inconvenient to be applied to the QDETS since it contains many product operations.  Thanks to the  systematic exposition of quaternion linear algebra (refer to \cite{zhang1997quaternions,wang2005general,rodman2014topics}) in recent years, there are quite a few accessible and significant results  can be applied not only by mathematicians, but also by scientists and engineers. In particular,  complex adjoint matrix representation of quaternion matrix plays a critical role in the current study. The  definition of determinant by  complex adjoint matrix, so-called $q$-determinant,  is crucial to define Wronskian of  QDETS and  to derive  the Liouville's formula for  QDETS.   Employing the newly established Wronskian and Liouville's formula for QDETS, we obtain the algebraic structure of general solutions of $n$ dimensional linear QDETS. It is a right quaternion module.

Explicit formulations of the fundamental  solution matrices (in particular, $e^{At}$) for the linear  QDEs with quaternion constant coefficient matrix  were  derived in  \cite{kou2015linear2}.  According to the discussion  in  \cite{baker1999right,farid2011eigenvalues}, the eigenvalue problem  of quaternion matrix is complicated.  A quaternion matrix usually has infinite number of eigenvalues. Moreover, the set of all eigenvectors corresponding to a non-real eigenvalue is not a submodule. If the
$n\times n$ coefficient  matrix $A$   has $n$ right linearly independent eigenvectors. Then  the  fundamental  solution matrix $e^{At}$ can be written in terms of the eigenvalues and eigenvectors. Otherwise,  more efforts   need to be  exerted. In  \cite{kou2015linear2}, the authors constructed $e^{At}$  by means of  series expansion and root subspace decomposition of quaternion matrix .

For linear  QDETS with quaternion constant coefficient matrix, it is not difficult  to find its fundamental  solution matrix if its coefficient matrix has enough  right linearly independent eigenvectors. Otherwise, we cannot use the method of combining series expansion and root subspace decomposition to construct the fundamental  solution matrices of  linear  QDETS. This is because that  the generalized exponential function on time scales does not possess simple series expansion  in contrast to $e^{At}$. In order to overcome this difficulty, we propose a modified Putzer's  algorithm to find the   fundamental  solution matrices of  linear  QDETS. The  Putzer's  algorithm  is particularly useful for quaternion  coefficient matrices  that do  not have enough  right linearly independent eigenvectors since it avoids the computing of  eigenvectors. To authors' best knowledge,  the   Vieta's formulas  of  quaternion polynomials and the  theory  of  annihilating polynomial of  quaternion matrices have not been well studied yet. Thus the  operability of  Putzer's  algorithm for    QDETS
may  be   confronted with some challenges. Still and all, further discussion  in a later section indicates that  the Putzer's  algorithm for    QDETS may after all be accepted as a good choice.

The rest of the paper is organized as follows.  In Section \ref{S2}, some basic concepts of quaternion algebra and the calculus of time scales are reviewed. In Section \ref{S3},  the first order linear  homogeneous QDETS are studied and the properties of generalized exponential function for QDETS  are investigated. In Section \ref{S4}, the structure of general solutions of higher order linear QDETS are analyzed. Specifically, a novel Wronskian determinant for QDETS  is  defined and the Liouville's formula and  variation of constants formula are given. In Section \ref{S5}, explicit formulations of the fundamental solution matrices  for linear
QDETS  with  constant coefficient matrix are  presented.  Some examples are given to illustrate the feasibility of the established  Putzer's algorithm. Finally, some conclusions are drawn at the end of the paper.

\section{Preliminaries}\label{S2}

\subsection{Quaternion  algebra}\label{S2.1}

The quaternions  were invented   in 1843 by Hamilton  \cite{sudbery1979quaternionic}.  The skew field of quaternions is denoted by
\begin{equation*}
	\H:= \{q=q_0+q_1\qi  +q_2 \qj+ q_3\qk\}
\end{equation*}
where $q_0,q_1,q_2,q_3$ are real numbers and the elements $\qi$, $\qj$ and $\qk$ obey the Hamilton's multiplication rules:
\begin{equation*}
	\qi\qj=-\qj\qi=\qk,~~\qj\qk=-\qk\qj=\qi,~~\qk\qi=-\qi\qk=\qj,~~\qi^2=\qj^2=\qi\qj\qk=-1.
\end{equation*}

For every quaternion $q=q_0+\qi q_1+\qj q_2+\qk q_3$, the scalar and vector parts of $q$, are  defined as $\Re(q)=q_0$ and $\Im(q)=q_1\qi  +q_2 \qj+ q_3\qk$, respectively. If $q= \Im(q)$, then $q$ is called pure imaginary quaternion.
The quaternion conjugate is defined by $\overline{q}=q_0-\qi q_1-\qj q_2-\qk q_3$, and the norm $|q|$ of $q$ is defined as
$|q|^2={q\overline{q}}={\overline{q}q}=\sum_{m=0}^{m=3}{q_m^2}$.
Using the conjugate and norm of $q$, one can define the inverse of $q\in\H\backslash\{0\}$ by $q^{-1}=\overline{q}/|q|^2$.
For each fixed unit pure imaginary quaternion $\bm \varsigma$, the quaternion has subset $\mathbb{C}_{\bm \varsigma}:=\{a+b \bm \varsigma  :a,b\in\mathbb{R}\}$  and  $\mathbb{C}_{\bm \varsigma}$ is   isomorphic to the complex numbers.

The quaternion exponential function $\exp(q)$ is defined by means of an infinite series as $$\exp(q):=\sum_{n=0}^\infty \frac{q^n}{n!}.$$
Analogous to the complex case one may derive a closed-form representation:
$$\exp(q)=\exp( \Re(q))\left(\cos|\Im(q)|+\frac{\Im(q)}{|\Im(q)|}\sin|\Im(q)|\right).$$
For simplicity of notations, we sometimes use $e^{q}$ to represent $\exp(q)$.
For every $q\in\H\backslash\{0\}$, its principal argument  is defined by
\begin{equation*}
	\Arg(q):=\arccos \frac{\Re(q)}{\abs{q}}\in [0,\pi].
\end{equation*}
Then all possible values of the argument  can be expressed as
$\arg (q)= \Arg(q) +2k\pi,~k\in \mathbb{Z}$.
It follows that   the polar form of a non-real  quaternion can be written as:
\begin{equation*}
	q=    \abs{q}\left(\frac{\Re(q)}{\abs{q}}+ \frac{  \Im(q) }{\abs{\Im(q)}} \cdot \frac{\abs{\Im(q)}}{\abs{q}}\right)
	= \abs{q} \exp ({\bm\varsigma }\theta),
\end{equation*}
where ${\bm\varsigma} =\frac{  \Im(q) }{\abs{\Im(q)}} $ and $\theta = \arg (q)$.
Accordingly,  the  principal  logarithm function is defined by
\begin{equation*}
	\Ln(q):=\begin{cases}\ln(\abs{q})+\frac{  \Im(q) }{\abs{\Im(q)}} \Arg(q)&q\in \H\setminus \mathbb{R},  \\
		\ln(\abs{q})+\qi\pi\frac{1-\sgn(q)}{2}&q\in \mathbb{R}.
	\end{cases}
\end{equation*}

Let $h>0$,  Georgiev and Morais \cite{georgiev2013introduction} introduced the Hilger quaternion numbers
\begin{equation*}
	\H_h:=\{p\in \H:p\neq -\frac{1}{h}\}.
\end{equation*}
They defined the addition $\oplus$ on $\H_h$  by
$ p \oplus q:= p+q+pqh$
and proved that $(\H_h, \oplus)$ is a group. The generalized  quaternion cylinder transformation was also given in \cite{georgiev2013introduction}:
\begin{equation}\label{quaternion cylinder transformation}
\xi_{h}(q):=\begin{cases}\frac{1}{h}\Ln(1+qh),&h> 0;\\
q,&h=0.\end{cases}
\end{equation}

Next we recall an important transformation between quaternion and complex matrices which were studied in  \cite{aslaksen1996quaternionic,zhang1997quaternions}. Every   quaternion
matrix $A\in\H^{m\times n}$
can be expressed uniquely in the form of
\begin{equation*}
	A=\cx_1(A)+\cx_2(A)\qj, ~~~\text{where}~~ \cx_1(A),  \cx_2(A) \in \C^{m\times n}.
\end{equation*}
So we can define  $\mathfrak{G}:\H^{m\times n}\to \C^{2m\times 2n}$ by
\begin{equation*}
	\mathfrak{G}(A):=
	\begin{pmatrix}
		\cx_1(A)&\cx_2(A)\\
		-\overline{\cx_2(A)} & \overline{\cx_1(A)}
	\end{pmatrix}.
\end{equation*}
where $\mathfrak{G}(A)$ is called the \emph{complex adjoint matrix} of the
quaternion matrix $A$. For simplicity, $\mathfrak{G}(A)$ will be denoted by $\chi_A$ in the following.

From \cite{kou2015linear1}, we know that $\H^n$ over the division ring $\H$ is a right $\H$-module and $\bmeta_1,\bmeta_2,\dots,\bmeta_k\in \H^n$ are right linearly independent if
\begin{equation*}
	\bmeta_1\alpha_1+\bmeta_2\alpha_2+\dots+\bmeta_k\alpha_k=0,\alpha_i\in \H~~ \text{implies that} ~~\alpha_1=\alpha_2=\dots=\alpha_k=0
\end{equation*}

Let $A\in\H^{n\times n}$, a   nonzero $\bmeta\in \H^{n\times 1}$ is said to be a \emph{right eigenvector} of $A$ corresponding to the \emph{right eigenvalue} $\lambda\in \H$  provided  that
\begin{equation*}
	A\bmeta=\bmeta \lambda
\end{equation*}
holds. A matrix $A_1$ is said to be similar to a matrix $A_2$ if $A_2=S^{-1}AS$ for some nonsingular matrix $S$. In particular, we say that  two quaternions $p,q$ are similar if $p=\alpha^{-1}q\alpha$ for some nonzero quaternion $\alpha$. We recall some basic results about   quaternion matrices which can be found, for instance, in \cite{zhang1997quaternions,baker1999right,rodman2014topics}.

\begin{theorem}\label{thm of q matrix}
	Let $A\in\H^{n\times n}$, then the following statements hold.
	\begin{enumerate}[(i)]
		\item  $A$  has exactly $n$ right eigenvalues (including multiplicity) which are  complex numbers with nonnegative imaginary parts. These eigenvalues are called standard eigenvalues.
		\item  If $\lambda$ is a eigenvalue of $A$, then  there exists a standard eigenvalue $\lambda'$ of $A$ such that $\lambda$ and $\lambda'$ are similar.
		\item $A$ is invertible if and only if $\chi_A$ is invertible.
		\item  $\det \chi_A \geq 0$, and the characteristic polynomial of $ \chi_A$ has real coefficients.
		\item  Let $\bmeta_1,\bmeta_2,\dots,\bmeta_k$  be eigenvectors of $A$ that correspond to eigenvalues  $\lambda_1,\lambda_2,\dots, \lambda_k$, respectively. If these  eigenvalues   are pairwise non-similar. Then $\bmeta_1,\bmeta_2,\dots,\bmeta_k$ are right linearly  independent.
		\item  If $A$ is (upper or lower) triangular, then the only eigenvalues are the diagonal elements (and the quaternions  similar to them).
	\end{enumerate}
\end{theorem}

\subsection{Calculus on time scales }\label{S2.2}

The theory of time scales  has gained much popularity in recent years.     Bohner and Peterson together with their research collaborators, such as Agarwal and Ahlbrandt,  have greatly developed
the theory of time scales (see e.g. \cite{agarwal1999sturm,ahlbrandt2000hamiltonian,peterson2006henstock,bohner2010laplace}). An systematic introduction to  dynamic equations on time scales was given by Bohner and Peterson \cite{bohner2001dynamic}.  We adopt the standard notations in
\cite{bohner2001dynamic,bohner2002advances,agarwal2002dynamic}. A time scale is a nonempty closed subset of $\mathbb{R}$. There are some typical examples of time scales.
\begin{enumerate}[(i)]
	\item  $\mathbb{R}$ consists of  all real numbers.
	\item  $h\mathbb{Z}:=\{hk:k\in \mathbb{Z}\}$, where $\mathbb{Z}$ is the set of integers.
	\item   $2^{N_0}:=\{2^k:k\in N_0\}$, where $N_0$ is the set of nonnegative integers.
	\item  $\mathbb{P}_{a,b}:=\displaystyle \bigcup_{k=0}^{\infty}\left[k(a+b),k(a+b)+a\right]$, where $a,b$ are positive real constants.
\end{enumerate}
Throughout the paper, let $\T$ be a time scale. For $t\in\T$, the forward jump operator $\sigma$ and the backward jump operator $\rho$ are respectively defined by
\begin{equation*}
	\sigma(t):=\inf\{s\in \T: s>t\} ~~~\text{and}~~~ \rho(t):=\sup\{s\in \T: s<t\}.
\end{equation*}
If $\sigma(t)>t, \sigma(t)=t,\rho(t)<t, \rho(t)=t$, then $t$ is said to be right-scattered, right-dense, left-scattered,  left-dense, respectively. The graininess function $\mu :\T \to [0,\infty)$ and the set $\T^{\kappa}$ are respectively defined by
\begin{equation}\label{graininess function}
\mu(t):=\sigma(t)-t
\end{equation}
and
\begin{equation*}
	\T^{\kappa} :=
	\begin{cases}
		\T \setminus (\rho(\sup\T),\sup\T],& \text{if}~~\sup\T< \infty ;\\
		\T,&\text{if}~~\sup\T=\infty .
	\end{cases}
\end{equation*}

The classical time scales calculus is only concerned with the real-valued functions. With minor adjustments, some basic concepts of  time scales calculus  can also   be  carried to  quaternion-valued functions. We denote the set of all quaternion-valued functions which are defined on time scales $\T$ by $\H\otimes \T$.

\begin{definition}
	Assume that $f\in \H\otimes \T$ and let $t\in \T^{\kappa}$. The delta derivative $f^{\Delta}(t)$ is defined to be the number (provided it exists) with property that given any $\varepsilon>0$, there exists $\delta>0$ such that
	\begin{equation*}
		\abs{f(\sigma(t))-f(s)-f^{\Delta}(t)(\sigma(t)-s)}\leq \varepsilon \abs{\sigma(t)-s}
	\end{equation*}
	holds for all  $s\in U_{\delta}:=(t-\delta,t+\delta)\cap \T$.
\end{definition}

By writing  $f\in \H\otimes \T$ in the form of $f(t)=f_0(t)+f_1(t)\qi+f_2(t)\qj+f_3(t)\qk$ with $f_i\in \mathbb{R}\otimes \T$, it is easy to verify that $f$ is delta differentiable if and only if  $f_0,f_1,f_2,f_3$ are  delta differentiable. Moreover, if  $f$ is delta differentiable, then
\begin{equation*}
	f^{\Delta}(t)=f_0^{\Delta}(t)+f_1^{\Delta}(t)\qi+f_2^{\Delta}(t)\qj+f_3^{\Delta}(t)\qk.
\end{equation*}
It follows that some useful results concerning the delta derivative for real-valued functions in \cite{bohner2001dynamic} can be carried to quaternion-valued functions.
\begin{theorem}\label{delta derivative for q-functions}
	Assume that $f,g \in \H\otimes \T$ are delta differentiable at $t\in\T^{\kappa}$,  then the following statements hold.
	\begin{enumerate}[(i)]
		\item $f(\sigma(t))=f(t)+\mu(t)f^{\Delta}(t)$.
		\item  $f+g$ is delta differentiable at $t$ and $(f+g)^{\Delta}(t)=f^{\Delta}(t)+g^{\Delta}(t)$.
		\item For any $\alpha,\beta\in \H$, $\alpha  f \beta$ is delta differentiable at $t$ and
		$(\alpha f\beta)^{\Delta}(t)=\alpha f^{\Delta}(t)\beta$.
		\item The product $fg$ is delta differentiable at $t$ and
		\begin{equation*}
			(fg)^{\Delta}(t)=f^{\Delta}(t)g(t)+f(\sigma(t))g^{\Delta}(t)=f(t)g^{\Delta}(t)+f^{\Delta}(t)g(\sigma(t)).
		\end{equation*}
		\item  If $f(t)f(\sigma(t))\neq 0$ then $ \widetilde{f}(t)=\left(f(t)\right)^{-1}$ is delta differentiable at $t$ and
		\begin{equation*}
			\widetilde{f}^{\Delta}(t)=-(f(\sigma(t)))^{-1}f^{\Delta}(t)(f(t))^{-1}=-(f(t))^{-1} f^{\Delta}(t)(f(\sigma(t)))^{-1}.
		\end{equation*}
	\end{enumerate}
\end{theorem}
\begin{remark}
	The statements 1,2,3 are easy to be understood.  The equality    $f^{\Delta}(t)g(t)+f(\sigma(t))g^{\Delta}(t)=f(t)g^{\Delta}(t)+f^{\Delta}(t)g(\sigma(t))$ can be seen from statement 1. Since $\H$ is noncommutative,  the equality $-(f(\sigma(t)))^{-1}f^{\Delta}(t)(f(t))^{-1}=-(f(t))^{-1} f^{\Delta}(t)(f(\sigma(t)))^{-1}$ is not obvious. But  this equality is true by invoking statement 4 (let  $g=\widetilde{f}$). We use an example to illustrate this result.
\end{remark}
\begin{example}
	Let $\T=\mathbb{Z}$, $f(t)=1+\qi+t \qj$. Then we have
	\begin{equation*}
		\begin{split}
			\widetilde{f}(t) =& (1-\qi-t\qj )(2+t^2)^{-1}, \\
			f(\sigma(t)) =  &  f(t+1)=1+\qi+(t+1)\qj,  \\
			f^{\Delta}(t) =& f(t+1)-f(t)=\qj.
		\end{split}
	\end{equation*}
	By direct computation,
	\begin{equation*}
		-(f(\sigma(t)))^{-1}f^{\Delta}(t)(f(t))^{-1}=-(f(t))^{-1} f^{\Delta}(t)(f(\sigma(t)))^{-1}
		=    \widetilde{f}^{\Delta}(t)
		=     \widetilde{f}(t+1)- \widetilde{f}(t).
	\end{equation*}
	They are equal to
	\begin{equation*}
		(-(2t+1)+(2t+1)\qi+(2-t^2-t)\qj)(2+t^2)^{-1}(2+(t+1)^2)^{-1}.
	\end{equation*}
\end{example}

To describe integrable quaternion-valued functions on time scales, we need to introduce the concept of rd-continuous. The rd-continuity of real-valued functions was defined by Bohner  \emph{et al.} \cite{bohner2001dynamic}.
\begin{definition}
	A real-valued function is called rd-continuous if it is continuous at right-dense points and its left-sided limits exist (finite) at left-dense points.
\end{definition}

Bohner  \emph{et al.} \cite{bohner2001dynamic}  proved that every rd-continuous function has an antiderivative. Next we introduce the rd-continuity and  integrability  of quaternion-valued functions.

We say that   $f=f_0+f_1\qi+f_2\qj+f_3\qk \in \H\otimes \T$ is rd-continuous provided that its every real components  $ f_0, f_1,f_2,f_3$  are rd-continuous. For every  rd-continuous function $f$, we define the integral by
\begin{equation*}
	\begin{split}
		\int_r^sf(t) \Delta t:= &   \int_r^sf_0(t) \Delta t+ \qi \int_r^sf_1(t) \Delta t  + \qj \int_r^sf_2(t) \Delta t+ \qk \int_r^sf_3(t) \Delta t\\
		= &  F_0(s)-F_0(r)+ (F_1(s)-F_1(r))\qi + (F_2(s)-F_2(r))\qj + (F_3(s)-F_3(r))\qk\\
		=&F(s)-F(r),
	\end{split}
\end{equation*}
where $F_i^{\Delta}(t)=f_i(t),  (0\leq i \leq 3)$   and $F^{\Delta}(t)= f(t)$ for $t\in \T^{\kappa}$.

From above discussion,  we have the following two theorems.
\begin{theorem}\label{prop of integral}
	If $f\in\H\otimes \T$ is rd-continuous and $t\in \T^{\kappa}$, then
	\begin{equation*}
		\int_t^{\sigma(t)} f(\tau)\Delta \tau=\mu(t)f(t).
	\end{equation*}
\end{theorem}
\begin{theorem}
	Let $a,b\in \T$ and suppose that $f\in \H\otimes \T$ is   rd-continuous.
	\begin{enumerate}
		\item If $\T=\mathbb{R}$, then
		\begin{equation*}
			\int_a^b f(t)\Delta t=\int_a^b f(t) dt.
		\end{equation*}
		Namely, it is the classical integral from calculus.
		\item  If $[a,b]\cap \T$ contains only isolated points, then
		\begin{equation*}
			\int_a^b f(t)\Delta t= \begin{cases}
				\sum_{t\in[a,b)\cap \T}\mu(t) f(t),& \text{if}~~a<b ;\\
				0,& \text{if}~~a=b ;\\
				- \sum_{t\in[b,a)\cap \T}\mu(t) f(t),&\text{if} ~~a>b .
			\end{cases}
		\end{equation*}
	\end{enumerate}
\end{theorem}

\section{First order linear QDETS}\label{S3}

In this section, we will study the first order linear QDETS and its corresponding initial value problems. Firstly, we need to introduce some auxiliary concepts.

\begin{definition}
	A function $p\in \H\otimes\T$ is said to be  regressive  if
	\begin{equation}\label{regressive}
	1+\mu(t)p(t)\neq 0,~~~ \text{for all}~~~t\in \T^\kappa.
	\end{equation}
\end{definition}

The set of all regressive and rd-continuous quaternion-valued functions is denoted by $\rc(\T,\H)$.
By similar arguments to  $(\H_h, \oplus)$ , we know that $\rc(\T,\H)$ is a group under addition $\oplus$ which is defined by
\begin{equation*}
	(p\oplus q)(t):=p(t)+q(t)+\mu(t)  p(t) q(t) ,~~~ \text{for all}~~~t\in \T^\kappa,
\end{equation*}
where $p,q\in \rc(\T,\H)$. Based on  the definition of quaternion cylinder transformation (\ref{quaternion cylinder transformation}), the generalized quaternion exponential function for $p \in \rc(\T,\H)$ is defined by
\begin{equation*}
	E_p(t,s):= \exp \left(\int_s^t\xi_{\mu(\tau)}(p(\tau))\Delta\tau\right).
\end{equation*}

Clearly, the generalized quaternion exponential function never be zero for any $p \in \rc(\T,\H)$. We proceed by presenting some important properties of  $ E_p(t,s)$.
\begin{lemma}\label{lemma of commute pro of exponential}
	If $p(t)=\alpha\in \H \setminus \mathbb{R}$ is a quaternion constant, then $p\in \rc(\T,\H)$ and $E_p(t,r) E_p(r,s)=E_p(t,s)$ for all $r,s,t\in \T$.
\end{lemma}
\begin{proof}
	Obviously, $p\in \rc(\T,\H)$. Let $r,s,t\in \T$, then both $\int^t_r \xi_{\mu(\tau)}(p(\tau))\Delta\tau$ and $\int^r_s \xi_{\mu(\tau)}(p(\tau))\Delta\tau$ are  $\mathbb{C}_{\bm\varsigma}$-valued, where ${\bm\varsigma}={ \frac{\mathfrak{I}(\alpha)}{\abs{\mathfrak{I}(\alpha)}}}$. Therefore
	\begin{equation*}
		\begin{split}
			E_p(t,r) E_p(r,s) & =\exp \left(\int_r^t\xi_{\mu(\tau)}(p(\tau))\Delta\tau\right)\exp \left(\int_s^r\xi_{\mu(\tau)}(p(\tau))\Delta\tau\right)\\
			& =\exp \left(\int_r^t\xi_{\mu(\tau)}(p(\tau))\Delta\tau + \int_s^r\xi_{\mu(\tau)}(p(\tau))\Delta\tau\right)\\
			&= \exp \left(\int_s^t\xi_{\mu(\tau)}(p(\tau))\Delta\tau\right)\\
			&= E_p(t,s)
		\end{split}
	\end{equation*}
	which completes the proof.
\end{proof}
\begin{lemma}\label{right-scattered exponential}
	Suppose that $\alpha\in \H \setminus \mathbb{R}$ and $\mu(t)>0$, then  $E_{\alpha}(\sigma(t),t)-1=\alpha \mu(t)$.
\end{lemma}
\begin{proof}
	Since  $\alpha\in \H \setminus \mathbb{R}$ and $\mu(t)>0$, then by Theorem
	\ref{prop of integral}
	\begin{equation*}
		\begin{split}
			E_{\alpha}(\sigma(t),t)&= \exp \left(\int_t^{\sigma(t)}\xi_{\mu(\tau)}(\alpha)\Delta\tau\right)   \\
			& = \exp \left( \mu(t)\xi_{\mu(t)}(\alpha) \right)  \\
			&=\exp \left(\Ln(1+\alpha \mu(t)) \right).
		\end{split}
	\end{equation*}
	Observe that $1+\alpha \mu(t)$ is  certainly not real-valued. It follows that  $$\exp \left(\Ln(1+\alpha \mu(t)) \right)=1+\alpha \mu(t),$$
	which completes the proof.
\end{proof}

Now we turn to study the following first order linear QDETS.
\begin{definition}
	If $p\in\rc(\T,\H)$, then the first order linear quaternion dynamic equation
	\begin{equation}\label{regressive equation}
	y^{\Delta}(t)=p(t)y(t)
	\end{equation}
	is called regressive. For any fixed $t_0 \in \T$ and $c_0\in \H$, the corresponding
	initial value problem is
	\begin{equation}\label{1D IVP}
	y^{\Delta}(t)=p(t)y(t),~~~y(t_0)=c_0.
	\end{equation}
\end{definition}

By similar arguments to Theorem 5.8 in \cite{bohner2001dynamic}, the initial value problem (\ref{1D IVP}) has exactly a unique solution.
\begin{remark}
	Let $\psi_p(t,t_0)$ denotes the unique solution of (\ref{1D IVP}) with $c_0=1$.
	In the classical case, we know that $\psi_p(t,s)=E_p(t,s)$. This assertion, however, is no longer true in the quaternion case.   Namely, $\psi_p(t,s)$ is not equal to $E_p(t,s)$ in general.
\end{remark}

\begin{example}\label{not exp}
	Consider the time scale $\T=\mathbb{Z}$. Let $p(t)=1+ {\qi} t+j$, then
	\begin{equation*}
		\begin{split}
			\xi_{\mu(\tau)}(p(\tau)) &=\frac{1}{\mu(\tau)}\Ln(1+p(\tau)) \\
			& = \frac{1}{2}\ln (5+\tau^2)+\frac{\qi \tau+\qj}{\sqrt{1+\tau^2}}\arccos \frac{2}{\sqrt{5+\tau^2}}.
		\end{split}
	\end{equation*}
	By direct computation, we have
	\begin{equation*}
		\int^1_0 \xi_{\mu(\tau)}(p(\tau)) \Delta \tau=\frac{1}{2}\ln5 + \qj \arccos\frac{2}{\sqrt{5}}.
	\end{equation*}
	Thus
	\begin{equation*}
		E_p(1,0)=\exp(\frac{1}{2}\ln5 + \qj \arccos\frac{2}{\sqrt{5}})=2+\qj.
	\end{equation*}
	Let $y(t)=E_p(t,0)$, we have
	\begin{equation*}
		y^{\Delta}(0)=y(1)-y(0)=2+\qj\neq 1+\qj=p(0)y(0),
	\end{equation*}
	which implies $\psi_p(t,0)\neq E_p(t,0)$.
\end{example}

Fortunately,  under some suitable conditions,   $\psi_p(t,s)$ is still equal to $E_p(t,s)$.
To prove this result, we give an useful lemma first.

\begin{lemma}\label{continuous of xi}
	If $\sigma(t)=t$ and $p(t)=\alpha\in \H\setminus \mathbb{R}$ is a quaternion constant function. Then
	$$\displaystyle \lim_{\tau\to t} \xi_{\mu(\tau)}(p(\tau))=\alpha. $$
\end{lemma}
\begin{proof}
	If there exists a number $\delta_0>0$ such that $\mu(\tau)=0$ for all $\tau \in U_{\delta_0}$. Then $$\displaystyle \lim_{\tau\to t} \xi_{\mu(\tau)}(p(\tau))=\lim_{\tau\to t}\xi_0 (\alpha)=\alpha.$$
	
	Otherwise,  for any $\delta>0$, there exists a number $\tau\in U_{\delta}$  such that $\mu(\tau)>0$. Let $\delta_1=1$, define
	\begin{equation*}
		D_1:=\{\tau \in U_{\delta_1}:\mu(\tau)>0\}~~~\text{and}~~~D_2:=\{\tau\in  U_{\delta_1}:\mu(\tau)=0 \},
	\end{equation*}
	then $D_1$ contains infinite elements. Now we claim that $\displaystyle \lim_{\tau\to t}\mu(\tau)=0=\mu(t)$. Since $\sigma(t)=t$,  then there exists a strictly decreasing  sequence $\{\tau_n\}\subset \T$ such that $\displaystyle \lim_{n\to \infty}\tau_n=t$. For any $\tau$ satisfying  $0< \tau-t<\tau_1$, there exists positive integer   $n_1$  such that $\tau_{n_1+1}>\tau>\tau_{n_1}$. Thus $\sigma(\tau)\leq \tau_{n_1}$. It follows that $\mu(\tau)=\sigma(\tau)-\tau \leq \tau_{n_1}-\tau\leq \tau_{n_1+1}-\tau_{n_1}$. Therefore
	\begin{equation*}
		\lim_{\tau\to t^+} \mu(\tau)=\lim_{n_1\to \infty}( \tau_{n_1+1}-\tau_{n_1})=0.
	\end{equation*}
	For any $\tau<t$, we have $\tau\leq \sigma(\tau)\leq t$. Then $\mu(\tau)=\sigma(\tau)-\tau\leq t- \tau$. Thus
	\begin{equation*}
		0\leq  \lim_{\tau\to t^-} \mu(\tau)\leq  \lim_{\tau\to t^-} (t-\tau)=0.
	\end{equation*}
	Then we have
	\begin{equation*}
		\begin{split}
			\lim_{{\tau\to t}\atop {\tau\in D_1}}\xi_{\mu(\tau)}(\alpha)&  = \lim_{{\tau\to t}\atop {\tau\in D_1}}\frac{1}{\mu(\tau)} \left (\ln\abs{1+\alpha \mu(\tau)}+\frac{\Im(\alpha)}{\abs{\Im(\alpha)}} \arccos\frac{1+\mu(\tau)\Re(\alpha)}{\abs{1+\alpha \mu(\tau)}}    \right)\\
			& =\Re(\alpha)+\frac{\Im(\alpha)}{\abs{\Im(\alpha)}}\abs{\Im(\alpha)}=\alpha
		\end{split}
	\end{equation*}
	by invoking the following two limits
	\begin{equation*}
		\lim_{h\to 0} \, \frac{\ln \left( \abs{1+h \Re(\alpha)} ^2+\abs{\Im(\alpha) h}^2\right)}{2 h}=\Re(\alpha),
	\end{equation*}
	\begin{equation*}
		\lim_{h\to 0} \, \arccos \frac{   1+h \Re(\alpha) }{ h \sqrt{ \abs{1+h \Re(\alpha)} ^2+\abs{\Im(\alpha) h}^2} }=\abs{\Im(\alpha)}.
	\end{equation*}
	In other words, for any $\varepsilon>0$, there exists positive number $\delta_2<\delta_1$ such that $\abs{\xi_{\mu(\tau)}(\alpha)-\alpha}<\varepsilon$ for all $\tau\in U_{\delta_2}\cap D_1$.
	Note that for all $\tau\in  U_{\delta_2}\cap D_2$, $\abs{\xi_{\mu(\tau)}(\alpha)-\alpha}=\abs{\xi_{0}(\alpha)-\alpha}=0<\varepsilon$. Hence
	$\abs{\xi_{\mu(\tau)}(\alpha)-\alpha}<\varepsilon$ holds for all $\tau \in  U_{\delta_2}$ which completes the proof.
\end{proof}
\begin{theorem}\label{exp of q-constant}
	If $p\in \rc(\T,\H)$ is real-valued or $p(t)=\alpha\in \H\setminus \mathbb{R}$ is a quaternion constant function. Then $\psi_p(t,s)=E_p(t,s)$ for all  $s,t \in \T$.
\end{theorem}
\begin{proof}
	For  real-valued $p\in \rc(\T,\H)$, please refer to Theorem 2.33 in \cite{bohner2001dynamic}. For  $p(t)=\alpha\in \H\setminus \mathbb{R}$, observe first that $E_{\alpha}(s,s)=1$.  We only need to show that $E_{\alpha}(t,s)$ satisfies (\ref{regressive equation}). Let $t\in\T^{\kappa}$ be right-scattered, that is, $\mu(t)=\sigma(t)-t>0$. By  applying Lemma \ref{lemma of commute pro of exponential} and  \ref{right-scattered exponential},
	\begin{equation*}
		\begin{split}
			E^{\Delta}_{\alpha}(t,s) & =\frac{E_{\alpha}(\sigma(t),s)-E_{\alpha}(t,s)}{\mu(t)} \\
			& =\frac{E_{\alpha}(\sigma(t),t)-1}{\mu(t)} E_{\alpha}(t,s) \\
			&=\alpha  E_{\alpha}(t,s).
		\end{split}
	\end{equation*}
	Let $t\in\T^{\kappa}$ be right-dense. For any given positive number $\varepsilon <\abs{E_{\alpha}(t,s)}$. Lemma  \ref{continuous of xi} implies that  there exists    $\delta_1>0$ such that
	\begin{equation*}
		\abs{\xi_{\mu(\tau)}(\alpha)-\alpha}<\frac{\varepsilon}{2 \abs{E_{\alpha}(t,s)}}<1
	\end{equation*}
	for all $\tau\in U_{\delta_1}$. Then
	$\abs{\int_{t'}^t\xi_{\mu(\tau)}(\alpha)\Delta\tau}\leq(1+\abs{\alpha})\abs{t-t'}$
	and
	\begin{equation*}
		\abs{E_{\alpha}(t,s)}\cdot \abs{\int_{t'}^t \left(\xi_{\mu(\tau)}(\alpha)-\alpha\right)\Delta\tau}\leq\frac{\varepsilon}{2} \abs{t-t'}
	\end{equation*}
	for all $t'\in  U_{\delta_1}$. Observe that
	$ \lim_{{q\to 0}\atop {q\in  \mathbb{C}_{\bm \zeta}}} \left(1-q-\exp(-q)\right)q^{-1}=0 $ where
	${\bm \zeta}=\frac{\Im(\alpha)}{\abs{\Im(\alpha)}}$. Then there exists  $\delta_1>\delta_2>0$ such that
	\begin{equation*}
		\abs{1-\int_{t'}^t\xi_{\mu(\tau)}(\alpha)\Delta\tau-E_{\alpha}(t',t)}\leq \frac{\varepsilon}{2(1+\abs{\alpha})\abs{E_{\alpha}(t,s)}}\cdot \abs{\int_{t'}^t\xi_{\mu(\tau)}(\alpha)\Delta\tau}
	\end{equation*}
	for all $t'\in  U_{\delta_2}$. Therefore, by Lemma \ref{lemma of commute pro of exponential}
	\begin{equation*}
		\begin{split}
			&  \abs{E_{\alpha}(t,s)-E_{\alpha}(t',s)-\alpha E_{\alpha}(t,s) (t-t') } \\
			=  & \abs{1-E_{\alpha}(t',t)-\alpha (t-t')}\cdot\abs{E_{\alpha}(t,s)}\\
			\leq & \abs{E_{\alpha}(t,s)}  \cdot \abs{\int_{t'}^t \left(\xi_{\mu(\tau)}(\alpha)-\alpha\right)\Delta\tau}\\
			&+ \abs{E_{\alpha}(t,s)}  \cdot\abs{1-\int_{t'}^t\xi_{\mu(\tau)}(\alpha)\Delta\tau-E_{\alpha}(t',t)}\\
			\leq &\frac{\varepsilon}{2} \abs{t-t'}+\frac{\varepsilon}{2} \abs{t-t'}\\
			=& \varepsilon \abs{t-t'}
		\end{split}
	\end{equation*}
	for all $t'\in  U_{\delta_2}$. This implies that $ E^{\Delta}_{\alpha}(t,s)= \alpha  E_{\alpha}(t,s)$. The proof is complete.
\end{proof}
\begin{proposition}\label{state transition 1D}
	Let $p \in \rc(\T,\H)$, then the following assertions hold.
	\begin{enumerate}[(i)]
		\item If $ \phi(t)$ satisfies (\ref{regressive equation}) and $\phi(t_0)=0$ for some $t_0\in\T$, then $\phi(t)=0$ for all $t\in\T$. In other words, if $\phi(t_0)\neq0$ for some $t_0\in\T$, then $\phi(t)\neq0$ for all $t\in\T$.
		\item If $ \phi(t)$ is a nonzero solution of (\ref{regressive equation}), then for any $t,r,s\in\T$, $\psi_p(t,s)=\phi(t)\phi^{-1}(s)=\psi_p(t,r)\psi_p(r,s)\neq 0$ and $\psi_p^{-1}(t,s)=\psi_p(s,t)$.
		\item For any given IVP $y(t_0)=c_0$, the solution is $\psi_p(t,t_0)c_0$.
		\item  $\psi_0(t,s)=1$ and  $\psi_p(t,t)=1$.
	\end{enumerate}
\end{proposition}
\begin{proof}
	In the classical case,  for any  $y=\phi(t)$ satisfying  (\ref{regressive equation}) never vanishes, because it is just a   composition of exponential  and  $\xi_{\mu(\cdot)}(p(\cdot))$. But Example \ref{not exp} indicates that  $\psi_p(t,s)$ is not necessarily to be $E_p(t,s)$ in the quaternion case. So assertion 1 is not a trivial result. Let $w(t)=\abs{\phi(t)}^2= \overline{\cx_1(\phi(t))}\cx_1(\phi(t)) + \overline{\cx_2(\phi(t))}\cx_2(\phi(t))$. Then by Theorem \ref{delta derivative for q-functions},
	\begin{equation*}
		\begin{split}
			w^{\Delta}  &=\overline{\cx_1(\phi)} \cx_1(\phi)^{\Delta}+\overline{\cx_1(\phi)} ^{\Delta} \cx_1(\phi)^{\sigma}+\overline{\cx_2(\phi)} \cx_2(\phi)^{\Delta}+\overline{\cx_2(\phi)} ^{\Delta} \cx_2(\phi)^{\sigma} \\
			&=  \overline{\cx_1(\phi)}\left(\cx_1(p)\cx_1(\phi)- \cx_2(p)\overline{\cx_2(\phi)} \right)+\left(\overline{\cx_1(p) \cx_1(\phi)}-\overline{\cx_2(p)}\cx_2(\phi)\right) \left( \cx_1(\phi)+\mu \cx_1(\phi)^{\Delta}\right)  \\
			&~~+ \overline{\cx_2(\phi)}\left(\cx_1(p)\cx_2(\phi)+ \cx_2(p)\overline{\cx_1(\phi)} \right)+\left(\overline{\cx_1(p) \cx_2(\phi)}+\overline{\cx_2(p)}\cx_1(\phi)\right) \left( \cx_2(\phi)+\mu \cx_2(\phi)^{\Delta}\right)   \\
			&=      2 \Re(p)\left(\abs{\cx_1(\phi)}^2+\abs{\cx_2(\phi)}^2\right) \\
			&~~~ +\mu\left(\overline{\cx_1(p) \cx_1(\phi)}-\overline{\cx_2(p)}\cx_2(\phi)\right) \left(\cx_1(p)\cx_1(\phi)- \cx_2(p)\overline{\cx_2(\phi)} \right) \\
			&~~~+ \left(\overline{\cx_1(p) \cx_2(\phi)}+\overline{\cx_2(p)}\cx_1(\phi)\right)\left(\cx_1(p)\cx_2(\phi)+ \cx_2(p)\overline{\cx_1(\phi)} \right)\\
			&= 2\left(\Re(p)+\mu \abs{p}^2\right)w.
		\end{split}
	\end{equation*}
	Therefore $w(t_0)=0$ implies that $w(t)= 0$ for all $t\in\T$ by Theorem \ref{exp of q-constant}. This completes the proof of assertion 1. The rest of assertions follows from  first assertion.
\end{proof}

\begin{example}
	Let $\T=2^{N_0}$ and $\alpha\in\H\setminus\mathbb{R}$ be a quaternion constant. Then $\mu(2^m)=2^{m+1}-2^m=2^m$ and therefore
	\begin{equation*}
		\xi_{\mu(2^m)}(\alpha) =2^{-m}\Ln(1+2^m \alpha).
	\end{equation*}
	It follows that
	\begin{equation*}
		\int_1^{2^k}\xi_{\mu(\tau)}(\alpha)\Delta \tau=\sum_{m=0}^{k-1}\mu(2^m)\xi_{\mu(2^m)}(\alpha)=\sum_{m=0}^{k-1}\Ln(1+2^m \alpha).
	\end{equation*}
	Thus
	\begin{equation*}
		E_{\alpha}(2^k,1)=\prod_{m=0}^{k-1}(1+2^m \alpha),
	\end{equation*}
\end{example}

\begin{example}\label{exponential in hZ}
	Let $\T=h\mathbb{Z}$ with $h>0$ and $\alpha\in\H\setminus\mathbb{R}$ be a quaternion constant. Then
	\begin{equation*}
		\begin{split}
			E_\alpha(t,0)= &  \abs{1+\alpha h}^{\frac{t}{h}}\left(\cos\left( \frac{t}{h}\arccos \frac{1+\alpha_0 h}{\abs{1+\alpha h}}\right)+  { \frac{\mathfrak{I}(\alpha)}{\abs{\mathfrak{I}(\alpha)}}}\sin \left( \frac{t}{h}\arccos \frac{1+\alpha_0 h}{\abs{1+\alpha h}}\right) \right) \\
			=  & \left(1+\alpha h \right)^{\frac{t}{h}}.
		\end{split}
	\end{equation*}
\end{example}

\section{ Linear systems of  QDETS}\label{S4}

Let $A(t)=(a_{ij}(t))_{m\times n}$ be  an  $m\times n$ quaternion-matrix-valued function   with $a_{ij}(t)\in \H\otimes \T$. We denote the set of such quaternion-matrix-valued functions by  $\H^{m\times n}\otimes \T$. We say  $A(t)$ is  rd-continuous (delta differentiable) on $\T$  if all of $a_{ij}(t)~(1\leq i\leq m, 1\leq j\leq n)$ is rd-continuous (delta differentiable) on $\T$. If $A(t)$ is delta differentiable on $\T$, we put $A^{\Delta}(t):=(a_{ij}^{\Delta}(t))_{n\times n}$.

In order to state our results, we introduce some notations which  are analogous to  those used in \cite{horn1985matrix,cormani2003liouville}. Let $\Omega_k^{(m)}$ be the set consisting of all possible $k$-combinations of the set $\{1,2,\cdot\cdot\cdot,m\}$. The number of elements of $\Omega_k^{(m)}$ is $ C_m^k$.  Let $\Lambda_k=\{i_1,i_2,\cdot\cdot\cdot,i_k \}\in \Omega_k^{(m)}$ be an index  set. For any $B\in \C^{m\times m}\otimes \T$, $B(\Lambda_k)$ denotes the principal sub-matrix that lies in the rows and columns of $B$ indexed by $\Lambda_k$ and
\begin{equation}\label{coefficient of chara polyn of chi-A}
\mathcal{V}_k(B):=\sum_{\Lambda_k\in \Omega_k^{(m)}}\det B(\Lambda_k)
\end{equation}
denotes the sum of determinants of  all $B(\Lambda_k), \Lambda_k\in \Omega_k^{(m)}$. Similarly, let $B(\Lambda_k,\Delta)$ be the $m\times m$ matrix generated from $B$  by replacing original entries with delta derivatives  on the rows indexed by $\Lambda_k$.

We consider the linear nonhomogeneous quaternion dynamic equations
\begin{equation}\label{nonhomegeneous eqn}
\bphi^{\Delta}(t){=}A(t)\bphi(t){+}\bfunc(t)
\end{equation}
and the linear homogeneous  quaternion dynamic equations
\begin{equation}\label{homegeneous eqn}
\bphi^{\Delta}(t)=A(t)\bphi(t)
\end{equation}
where $\bphi(t), \bfunc(t)\in\H^{n\times 1}\otimes \T$ and $A(t)\in \H^{n\times n}\otimes \T$. Let $t_0\in\T$ and $\bmeta\in\H^n$,   then the corresponding initial value problem is
\begin{equation}\label{IVP}
\bphi(t_0)=\bmeta.
\end{equation}

We are in a position to introduce the concept of  regressivity of $A(t)$.
\begin{definition}
	We say   $A(t)\in \H^{n\times n}\otimes \T$ is regressive  provided that
	\begin{equation}
	I_n+\mu(t)A(t) ~~~  \text{is invertible for all}~~~t\in\T^{\kappa}.
	\end{equation}
	The totality
	of  all such regressive and rd-continuous quaternion-valued functions is denoted by $\mathcal{R}(\T,\H^{n\times n})$. If $A\in \mathcal{R}(\T,\H^{n\times n})$, we say system (\ref{nonhomegeneous eqn}) is regressive.
\end{definition}

\begin{lemma}\label{equivalent conditions of  regressivity 1}
	Let  $A(t)\in \H^{n\times n}\otimes \T$. Then for any fixed $t\in\T$, $I_n+\mu(t)A(t)$ is invertible if and  only if $I_{2n}+\mu(t)\chi_A(t)$ is invertible.
\end{lemma}
\begin{proof}
	Let $\widetilde{A}(t)=I_n+\mu(t)A(t)$, it is easy to see that $\chi_{\widetilde{A}}(t)=I_{2n}+\mu(t)\chi_{A}(t)$. Then by statement 3 of Theorem
	\ref{thm of q matrix}, we complete the proof.
\end{proof}

We need a lemma about the  equivalent conditions of  regressivity   in the classical case.  This result  can be found in \cite{bohner2001dynamic,cormani2003liouville}.

\begin{lemma}\label{equivalent conditions of  regressivity 2}
	Let $B(t)\in \C^{m\times m}\otimes \T$, then for any fixed $t\in\T$, the following statements are equivalent.
	\begin{enumerate}[(i)]
		\item  $I_m+\mu(t)B(t)$ is invertible.
		\item   The eigenvalues $\lambda_i(t)$ of $B(t)$ is regressive for all $1\leq i \leq m$. Namely, $1+\mu(t)\lambda_i(t)\neq 0$ for all $1\leq i \leq m$.
		\item  $u(t)$ is regressive, namely, $1+\mu(t)u(t)\neq 0$, where
		$u(t)=\sum_{k=1}^{2n}\mu^{k-1}(t)\mathcal{V}_k(B(t))$ and $\mathcal{V}_k$ is defined by (\ref{coefficient of chara polyn of chi-A}).
	\end{enumerate}
\end{lemma}

An immediate consequence of  Lemma \ref{equivalent conditions of  regressivity 1} and
\ref{equivalent conditions of  regressivity 2}    is the   equivalent conditions of  regressivity  for $A(t)\in \H^{n\times n}\otimes \T$.

\begin{theorem}\label{equivalent conditions of  regressivity 3}
	Let  $A(t)\in \H^{n\times n}\otimes \T$. Then for any fixed $t\in\T$, $A(t)$ is  regressive if and  only if all of its eigenvalues  are  regressive.
\end{theorem}
\begin{proof}
	By Lemma \ref{equivalent conditions of  regressivity 1} and
	\ref{equivalent conditions of  regressivity 2},  we know that   $A(t)\in \H^{n\times n}\otimes \T$ is regressive if and only if all the standard eigenvalues  of $A(t)$ are regressive.  Note that if $\lambda\in\H$ is  regressive, then
	$\alpha^{-1} \lambda \alpha$ is also regressive for any nonzero $\alpha\in \H$. Then by statement 2 of Theorem \ref{thm of q matrix}, we complete the proof.
\end{proof}

By similar arguments to Theorem 5.8 in \cite{bohner2001dynamic}, we have the following existence and uniqueness theorem.

\begin{theorem}
	If $A\in \mathcal{R}(\T,\H^{n\times n})$ and $\bfunc$ is rd-continuous. Then the initial value problem (\ref{nonhomegeneous eqn}), (\ref{IVP}) has exactly a unique solution.
\end{theorem}

To study the properties of solutions of (\ref{homegeneous eqn}), we should define the concept of \emph{Wronskian} for quaternion dynamic equations. Due to the noncommutative property of quaternons,  there is no unified definition of determinant of quaternion matrix. Many researchers have proposed different definitions. But, as mentioned in \cite{kou2015linear1}, some definitions of determinant  may be  not suitable to  define   {Wronskian}. Kou  \emph{et al.}  \cite{kou2015linear2}   adopted the definition of determinant based on permutation proposed  by Chen \cite{longxuan1991definition}. The proof of Liouville's formula in \cite{kou2015linear2} is complicated.
In this paper, we adopt an  alternative definition of determinant called \emph{q-determinant} \cite{zhang1997quaternions} for $A\in \H^{n\times n}$.  It is defined by
\begin{equation}\label{q determinant}
\qd A:=\det  \chi_A.
\end{equation}

We know that the product rule  for delta derivative is more tedious than  traditional one. But the Wronskian defined in   \cite{kou2015linear2}  contains many product operations. So we use $q$-determinant to define the Wronskian for quaternion dynamic equations on time scales.

\begin{definition}
	Let $\bx_1(t), \bx_2(t), \cdot\cdot\cdot,\bx_n(t)$, $\bx_i(t)\in \H^n \otimes \T$ be $n$ quaternion-vector-valued functions. Denote
	\begin{equation*}
		M(t):=(\bx_1(t),\bx_2(t), \cdot\cdot\cdot,\bx_n(t))= \begin{pmatrix}
			x_{11}(t)&x_{12}(t)&\cdot\cdot\cdot& x_{1n}(t)\\
			x_{21}(t)&x_{22}(t)&\cdot\cdot\cdot& x_{2n}(t)\\
			~&{\cdot\cdot\cdot}&~&~\\
			x_{n1}(t)&x_{n2}(t)&\cdot\cdot\cdot& x_{nn}(t)
		\end{pmatrix}.
	\end{equation*}
	The Wronskian of $M(t)$   is defined by
	\begin{equation*}
		W(t):=\qd M(t)=\det  \chi_M(t).
	\end{equation*}
	If  $\bx_1(t), \bx_2(t), \cdot\cdot\cdot,\bx_n(t)$ are  solutions  of
	(\ref{homegeneous eqn}), then  we  call $M(t)$  a solution matrix of (\ref{homegeneous eqn}).
\end{definition}
\begin{theorem}\label{superposition theorem}
	If $\bx_1(t), \bx_2(t) \in \H^n \otimes \T$ are solutions of   (\ref{homegeneous eqn}), then  the    right  linear combination  $\bx_1(t)\alpha_1+\bx_2(t)\alpha_2$ for any $\alpha_1,\alpha_2\in \H$ is also a solution of (\ref{homegeneous eqn}).
\end{theorem}
\begin{theorem}\label{dependent solutions}
	If  $\bx_1(t), \bx_2(t), \cdot\cdot\cdot,\bx_n(t)$, $\bx_i(t)\in \H^n \otimes \T$ are right linearly dependent on $\T$, then $W(t)= 0$ for all $t\in\T$. In other words, if  $W(t)\neq 0$ for some $t_0\in \T$, then $x_1(t),x_2(t),\cdot\cdot\cdot, x_n(t)$ are right linearly independent on $\T$.
\end{theorem}
\begin{proof}
	If $\bx_1(t), \bx_2(t), \cdot\cdot\cdot,\bx_n(t)$  are right linearly dependent on $\T$, then there exists a nonzero vector $\bq=(q_1,q_2,\cdot\cdot\cdot, q_n)^{\top}\in \H^n$ such that
	\begin{equation*}
		M(t)\bq=0, ~~~\text{for all} ~~~t\in \T.
	\end{equation*}
	If there exists $t_0\in \T$ such that $W(t_0)\neq 0$,   by definition of Wronskian, we have  $\chi_M(t_0)$ is  invertible. From Theorem \ref{thm of q matrix}, it follows that $M(t_0)$ is  invertible. Thus $M(t_0)\bmeta =0$ has a unique solution $\bmeta =0$ by Theorem 4.3 of \cite{zhang1997quaternions}. This contradicts the fact that $\bmeta=\bq$ is a nonzero solution of $M(t_0)\bmeta =0$.
\end{proof}
\begin{theorem}\label{independent solutions}
	If $\bx_1(t), \bx_2(t), \cdot\cdot\cdot,\bx_n(t)$  are $n$ right linearly independent solutions of (\ref{homegeneous eqn}), then $W(t)  \neq 0$ for all  $t\in \T$.
\end{theorem}
\begin{proof}
	Assume that there exists $t_0\in \T$ such that $W(t_0)=0$. By similar arguments to Theorem \ref{dependent solutions}, it is easy to see that there is  a nonzero vector $\bq=(q_1,q_2,\cdot\cdot\cdot, q_n)^{\top}\in \H^n$ such that $M(t_0)\bq=0$. Define $\bx(t):=M(t)\bq$ for all $t\in \T$. By Theorem \ref{superposition theorem}, $\bx(t)$ is a solution of  (\ref{homegeneous eqn}) with $\bx(t_0)=0$. Note that $\by(t)\equiv 0$ is the unique solution of  (\ref{homegeneous eqn}) and initial condition (\ref{IVP}) with $\bmeta=0$. Therefore $ M(t)\bq=\bx(t)=\by(t)\equiv 0$. This implies that $\bx_1(t), \bx_2(t), \cdot\cdot\cdot,\bx_n(t)$  are   right linearly  dependent, which is contradiction to the hypotheses of the theorem.
\end{proof}

To deduce the Liouville's formula, we need an important Lemma.
\begin{lemma}\label{lemma of Liou}
	Let $M(t)$ be a solution matrix of (\ref{homegeneous eqn}) and $W(t)$ be the corresponding Wronskian.  Then for any index set $\Lambda_k\in \Omega_k^{(m)}$ with $m=2n$, we have
	\begin{equation*}
		\det  \chi_M(\Lambda_k,\Delta)=\det \chi_A(\Lambda_k)\cdot W
	\end{equation*}
\end{lemma}
\begin{proof}
	Without loss of generality,  we may assume $\Lambda_k=\{i_1,i_2,\cdot\cdot\cdot,i_k\}$ with $1\leq i_1<i_2 < \cdot\cdot\cdot < i_k \leq m$.  We only consider the case of $i_k\leq n$, the other cases can be similarly proved.
	
	Observe that $M^{\Delta}(t)=A(t)M(t)$, then $(i_r,j)$-component of  $\chi_M(\Lambda_k,\Delta)$ is
	\begin{equation*}
		\begin{cases}\displaystyle \sum_{s=1}^n\cx_1(a_{i_r s})\cx_1(x_{sj})-\cx_2(a_{i_r s})\overline{\cx_2(x_{sj})},&  j  \leq  n;\\
			\displaystyle  \sum_{s=1}^n\cx_1(a_{i_r s})\cx_2(x_{sj})+\cx_2(a_{i_r s})\overline{\cx_1(x_{sj})},&j>n.\end{cases}
	\end{equation*}
	If  $ i \leq n$ and  $i\notin \Lambda_k$, the    $(i,j)$-component of $\chi_M(\Lambda_k,\Delta)$   is
	\begin{equation*}
		\begin{cases}\displaystyle  \cx_1(x_{ij}) ,&  j  \leq  n;\\
			\displaystyle   \cx_2(x_{ij}) ,&j>n.\end{cases}
	\end{equation*}
	If $i>n$, the   $(i,j)$-component of $\chi_M(\Lambda_k,\Delta)$   is
	\begin{equation*}
		\begin{cases}\displaystyle - \overline{\cx_2(x_{i-n,j})} ,&  j  \leq  n;\\
			\displaystyle  \overline{ \cx_1(x_{i-n,j})} ,&j>n.\end{cases}
	\end{equation*}
	Then for  $ s \leq n$ and  $s\notin \Lambda_k$, we  do row operation of  $- \cx_1(a_{i_r s}) R_s+R_{i_r}$ for $\chi_M(\Lambda_k,\Delta)$. And for $ s >n$, we do row operation of $-\cx_2(a_{i_r, s-n})R_s+R_{i_r}$ for $\chi_M(\Lambda_k,\Delta)$. After doing this procedure for $1\leq r \leq k$, we obtain a new matrix $\chi_M(\Lambda_k,\Delta,new)$. The $(i_r,j)$-component of  $\chi_M(\Lambda_k,\Delta,new)$ is
	\begin{equation*}
		\begin{cases}\displaystyle \sum_{s=1}^k\cx_1(a_{i_r i_s})\cx_1(x_{i_sj}), &  j  \leq  n;\\
			\displaystyle  \sum_{s=1}^k\cx_1(a_{i_r i_s})\cx_2(x_{i_sj}), &j>n.\end{cases}
	\end{equation*}
	Construct a block diagonal matrix
	\begin{equation*}
		S:=\begin{pmatrix}
			I_{i_1-1}&0&0\\
			0 &H& 0\\
			0&0&I_{m-i_k}
		\end{pmatrix},
	\end{equation*}
	where $I_{s}$ is the $s$-order  identity matrix and
	\begin{equation*}
		H:=
		\begin{pmatrix}
			\cx_1(a_{i_1i_1})&0& \cx_1(a_{i_1i_2})&  0 & \cx_1(a_{i_1i_3})&\dots &  \cx_1(a_{i_1i_k})\\
			0& I_{i_2-i_1-1} & 0&  0 & 0&\dots &  0\\
			\cx_1(a_{i_2i_1})&0& \cx_1(a_{i_2i_2})&  0 & \cx_1(a_{i_2i_3})&\dots &  \cx_1(a_{i_2i_k})\\
			0& 0 & 0&  I_{i_3-i_2-1} & 0&\dots &  0 \\
			\cx_1(a_{i_3i_1})&0& \cx_1(a_{i_3i_2})&  0 & \cx_1(a_{i_3i_3})&\dots &  \cx_1(a_{i_3i_k})\\
			\vdots&\vdots&\vdots&\vdots&\vdots &\ddots&\vdots\\
			\cx_1(a_{i_ki_1})&0& \cx_1(a_{i_ki_2})&  0 & \cx_1(a_{i_ki_3})&\dots &  \cx_1(a_{i_ki_k})
		\end{pmatrix}.
	\end{equation*}
	Then $\chi_M(\Lambda_k,\Delta,new)=S \chi_M$. Therefore
	\begin{equation*}
		\det  \chi_M(\Lambda_k,\Delta)=\det \chi_M(\Lambda_k,\Delta,new)= \det S \cdot\det \chi_M=(\det H )\cdot W.
	\end{equation*}
	Observe that $\det H= \det \chi_A(\Lambda_k)$,  the proof is complete.
\end{proof}

We need a lemma about delta derivative of determinant  from Theorem 5.105 in \cite{bohner2001dynamic}.
\begin{lemma}\label{delta derivative of determinant}
	\cite{bohner2001dynamic}   Let $C=(c_{ij})_{1\leq i, j\leq m }\in \mathbb{R}^n\otimes \T$ be delta differentiable. For $1\leq k\leq m$, define $B^{(k)}=(b^{(k)}_{ij})_{1\leq i, j\leq m}$ by
	\begin{equation*}
		b^{(k)}_{ij}:=  \begin{cases}
			c^{\sigma}_{ij} & ~\text{if}~~~i<k\\
			c^{\Delta}_{ij} &~\text{if}~~~i=k\\
			c_{ij} &~\text{if}~~~i>k.
		\end{cases}
	\end{equation*}
	Then $\det C \in \mathbb{R}\otimes \mathbb{T}$ is delta differentiable, and
	\begin{equation*}
		(\det C)^{\Delta}=\sum_{k=1}^m\det B^{(k)}.
	\end{equation*}
\end{lemma}
\begin{remark}
	Although Lemma \ref{delta derivative of determinant} is established for real-valued functions, it is easy to verify that Lemma \ref{delta derivative of determinant} is also valid   for complex-valued functions.
\end{remark}

Now we present the Liouville's formula.
\begin{theorem}
	The Wronskian of solution matrix $M(t)$  of  (\ref{homegeneous eqn}) satisfies the following  Liouville's formula.
	\begin{equation}\label{Liouville}
	W(t)=E_u(t,t_0)W(t_0)
	\end{equation}
	where $u=\sum_{k=1}^{2n}\mu^{k-1}\mathcal{V}_k(\chi_A)$ and $\mathcal{V}_k$ is defined by (\ref{coefficient of chara polyn of chi-A}).
\end{theorem}
\begin{proof}
	From Lemma \ref{delta derivative of determinant}, we see that
	\begin{equation*}
		\begin{split}
			W^{\Delta} & =(\det \chi_M)^{\Delta} \\
			& =\det B^{(1)}+\det B^{(2)}+\dots+ \det B^{(2n)}\\
			&=\det
			\begin{pmatrix}
				\chi_M^{\Delta}(1,:) \\
				\chi_M(2,:) \\
				\vdots \\
				\chi_M(2 n,:)
			\end{pmatrix}+
			\det
			\begin{pmatrix}
				\chi_M^{\sigma}(1,:) \\
				\chi_M^{\Delta}(2,:) \\
				\vdots \\
				\chi_M(2 n,:)
			\end{pmatrix}+\dots+
			\det
			\begin{pmatrix}
				\chi_M^{\sigma}(1,:) \\
				\chi_M^{\sigma}(2,:) \\
				\vdots \\
				\chi_M^{\Delta}(2 n,:)
			\end{pmatrix},
		\end{split}
	\end{equation*}
	where
	\begin{equation*}
		B^{(s)}=\begin{pmatrix}
			\chi_M^{\sigma}(1,:) \\
			\chi_M^{\sigma}(2,:) \\
			\vdots \\
			\chi_M^{\Delta}(s,:)\\
			\vdots \\
			\chi_M(2 n,:)
		\end{pmatrix}
	\end{equation*}
	and $\chi_M(j,:)$, $(1\leq j\leq 2n)$ is the $j$-th row of  $\chi_M$. Observe that
	\begin{equation*}
		\chi_M^{\sigma}(k,:)=\chi_M (k,:)+\mu \chi_M^{\Delta}(k,:) ~\text{for} ~ 1\leq k< s.
	\end{equation*}
	Thus
	\begin{equation*}
		\det B^{(s)}=\sum_{k=1}^{s}\sum_{{\Lambda_k\in \Omega_k^{(2n)}} \atop{\max \Lambda_k=s}}\mu^{k-1}\det \chi_M(\Lambda_k,\Delta)
	\end{equation*}
	By Lemma \ref{lemma of Liou}, we obtain
	\begin{equation*}
		\begin{split}
			W^{\Delta}  =\sum_{s=1}^{ 2n }\det B^{(s)}& =\sum_{s=1}^{2n}\sum_{k=1}^{s}\sum_{{\Lambda_k\in \Omega_k^{(2n)}} \atop{\max \Lambda_k=s}}\mu^{k-1}\det \chi_A(\Lambda_k)\cdot W\\
			&=\sum_{k=1}^{2n}\sum_{s=k}^{2n}\sum_{{\Lambda_k\in \Omega_k^{(2n)}} \atop{\max \Lambda_k=s}}\mu^{k-1}\det \chi_A(\Lambda_k)\cdot W \\
			&=\sum_{k=1}^{2n} \sum_{{\Lambda_k\in \Omega_k^{(2n)}}}\mu^{k-1}\det \chi_A(\Lambda_k)\cdot W\\
			&=uW.
		\end{split}
	\end{equation*}
	In fact, $\mathcal{V}_1(\chi_A),\mathcal{V}_2(\chi_A), \cdot\cdot\cdot, \mathcal{V}_{2n}(\chi_A)$  are  coefficients of  the characteristic polynomial of $ \chi_A$. Then, by  Theorem
	\ref{thm of q matrix} and the definition of $u$, we conclude  that  $u$ is real-valued.  From Theorem \ref{exp of q-constant} follows   the  Liouville's formula (\ref{Liouville}).
\end{proof}

\begin{remark}
	If $\T=\mathbb{R}$, then the graininess function $\mu$  vanishes identically. Thus
	$$u(t)=\mathcal{V}_1(\chi_A(t))= tr\chi_A(t)=2 \Re \left(trA(t)\right).$$
	Therefore     the  Liouville's formula  becomes
	\begin{equation*}
		W(t)=\exp \left( 2\int_{t_0}^t \Re(trA(\tau))d\tau\right)W(t_0).
	\end{equation*}
	If  $\T=\mathbb{Z}$,  then the graininess function $\mu$   is identically equal to $1$. Therefore
	$u(t)$ equals to the sum of all coefficients of the characteristic polynomial of $\chi_A(t)$. In this case,  the  Liouville's formula  becomes
	\begin{equation*}
		W(t)=\exp\left(  \sum_{\tau=t_0}^{t-1}\xi_1(u(\tau))\right)W(t_0).
	\end{equation*}
	
\end{remark}

\begin{corollary}\label{prop of wronskian}
	Let $W(t)$ be the Wronskian of solution matrix $M(t)$  of  (\ref{homegeneous eqn}). Then $W(t)=0$  at same $t_0\in \T$ if and only if $W(t)=0$ for all $t\in \T$.
\end{corollary}

As an immediate consequence of Theorem \ref{dependent solutions}, \ref{independent solutions} and Corollary \ref{prop of wronskian}, we have
\begin{theorem}\label{independent or dependent}
	Let  $A\in \mathcal{R}(\T,\H^{n\times n})$   and  $W(t)$ be the Wronskian of solution matrix $M(t)$  of  (\ref{homegeneous eqn}). Then $\bx_1(t),\bx_2(t),\cdot\cdot\cdot, \bx_n(t)$ are right linearly  dependent on $\T$ if and only if $W(t)=0$ at some $t_0\in T$. And  $\bx_1(t),\bx_2(t),\cdot\cdot\cdot, \bx_n(t)$ are right linearly independent on $\T$ if and only if $W(t)\neq0$ at some $t_0\in T$.
\end{theorem}

Let $A\in \mathcal{R}(\T,\H^{n\times n})$  and  $ \bme_k$ be the  $k$-th column of $I_n$. Then there exists a unique solution $\bx_k(t)$ of (\ref{homegeneous eqn}) satisfying $\bx_k(t_0)=\bme_k$.
Thus $M(t)=(\bx_1(t), \bx_2(t), \cdot\cdot\cdot,\bx_n(t))$  is a solution matrix of  (\ref{homegeneous eqn}). Let $W(t)$ be the Wronskian of $M(t)$. It is easy to see that
$W(t_0)=1\neq0$. Then we immediately have
\begin{theorem}\label{fundamental matrix}
	If $A\in \mathcal{R}(\T,\H^{n\times n})$, then (\ref{homegeneous eqn}) has exactly  $n$ right linearly independent solutions $\bx_1(t), \bx_2(t), \cdot\cdot\cdot,\bx_n(t)$. The corresponding solution matrix $M(t)$ is called fundamental solution matrix of
	(\ref{homegeneous eqn}).
\end{theorem}

Theorem \ref{fundamental matrix} illustrates the existence of fundamental solution matrix. In fact, fundamental solution matrix is not unique. Suppose that $M(t)$ is a fundamental solution matrix of (\ref{homegeneous eqn}). Multiplying $M(t)$ with an arbitrary non-singular quaternion matrix on the right-side, it remains to be a fundamental solution matrix.  Using Theorem \ref{independent or dependent},  we can easily determine whether $M(t)$ is a  fundamental solution matrix or not.

\begin{corollary}\label{whether fundamental solution matrix or not}
	A solution matrix $M(t)$ of (\ref{homegeneous eqn}) is a fundamental solution matrix if and only if its corresponding Wronskian $W(t)\neq 0$ for some $t_0\in\T$. Moreover, if $M(t)$ is a fundamental solution matrix,  then   $M(t)$  is invertible for all  $t\in \T$.
\end{corollary}

Next result describes the structure of the general solution of (\ref{homegeneous eqn}). It is a right $\H$-module. That implies that if we know $n$ right linearly independent
solutions of (\ref{homegeneous eqn}), then we actually know
all possible solutions of  (\ref{homegeneous eqn}), since any other solution is just a right linear combination of known solutions.

\begin{theorem}\label{structure of the general solution}
	Let $M(t)=(\bx_1(t), \bx_2(t), \cdot\cdot\cdot,\bx_n(t))$  be a fundamental solution matrix of  (\ref{homegeneous eqn}). Then any solution of (\ref{homegeneous eqn}) can be written as
	\begin{equation}\label{general sol}
	\bx(t)=M(t)\bq
	\end{equation}
	where $\bq=(q_1,q_2,\cdot\cdot\cdot,q_n)^{\top}\in \H^n$ is undetermined quaternion vector. The totality of the solutions form a right $\H$-module.
\end{theorem}
\begin{proof}
	For any $t_0\in\T$, by  Corollary \ref{whether fundamental solution matrix or not}, $M(t_0)$ is  invertible. Thus $M(t_0)\bmeta=\bx(t_0)$ has a unique solution $\bmeta=\bq$. Note that $M(t)\bq$ is also a solution of \ref{homegeneous eqn} with initial condition $\bphi(t_0)=\bx(t_0)$. By the  uniqueness theorem,  the equality  (\ref{general sol}) holds.
\end{proof}

Let $A\in \mathcal{R}(\T,\H^{n\times n})$ and $M(t)$  be a fundamental solution matrix of  (\ref{homegeneous eqn}). Define the state-transition matrix
$$\Psi_A(t,s):=M(t)M^{-1}(s).$$
It is not difficult to verify that $\Psi_A(t,s)$ is well-defined. Suppose that $M_1(t)$ be   a fundamental solution matrix   which is  different from  $M(t)$. It is easy to see that $M_1 (t) M^{-1}_1(s)=M  (t) M^{-1}(s)$ by uniqueness theorem.    By the similar arguments  to Proposition \ref{state transition 1D}, we have
\begin{proposition}
	Let $A \in \mathcal{R}(\T,\H^{n\times n})$ and $t,r,s\in\T$, then the following assertions hold.
	\begin{enumerate}[(i)]
		\item   $\Psi_A(t,s)$    is invertible.
		\item  $\Psi_0(t,s)\equiv I_n$  and  $\Psi_A(t,t)=I_n$.
		\item  $\Psi_A(t,s)\Psi_A(s,r)=\Psi_A(t,r)$. In particular, $\Psi_A^{-1}(t,s)=\Psi_A(s,t)$.
		\item  Any solution $\bx(t)$ of (\ref{homegeneous eqn})  with  $\bx(t_0)=\bmeta_0$ can be expressed by $\bx(t)=\Psi_A(t,t_0)\bmeta_0$.
	\end{enumerate}
\end{proposition}

After studying the homogeneous equations (\ref{homegeneous eqn}), we also consider the nonhomogeneous equations (\ref{nonhomegeneous eqn}).  It is easy to verify  the following two results.
\begin{lemma}
	If $\bphi^{NH}(t)$ and $\bphi^{H}(t)$ are solutions of   (\ref{nonhomegeneous eqn})     and (\ref{homegeneous eqn}) respectively. then $\bphi^{NH}(t)+\bphi^{H}(t)$ is a solution of  (\ref{nonhomegeneous eqn}).
\end{lemma}
\begin{lemma}\label{diff of two nonhomo eqn}
	If $\bphi_1^{NH}(t)$ and $\bphi_2^{NH}(t)$ are two  solutions of   (\ref{nonhomegeneous eqn}).     then $\bphi_1^{NH}(t)-\bphi_2^{NH}(t)$ is a solution of  (\ref{homegeneous eqn}).
\end{lemma}

Next result  describes the structure of the general solution of (\ref{nonhomegeneous eqn}).

\begin{theorem}\label{Thm of structure of nonhomo eqn}
	Suppose that $M(t)$ be a  fundamental solution matrix of (\ref{homegeneous eqn}) and $\bphi_0^{NH}(t)$ be a solution of (\ref{nonhomegeneous eqn}). Then any solution of  (\ref{nonhomegeneous eqn}) can be expressed by
	\begin{equation}\label{structure of nonhomo eqn}
	\bphi^{NH}(t)=M(t)\bq+\bphi_0^{NH}(t)
	\end{equation}
	where $\bq=(q_1,q_2,\cdot\cdot\cdot,q_n)^{\top}$ is a constant quaternion vector.
\end{theorem}
\begin{proof}
	For any solution $\bphi^{NH}(t)$ of (\ref{nonhomegeneous eqn}), we know that $\bphi^{NH}(t)-\bphi_0^{NH}(t)$ is a solution of (\ref{homegeneous eqn}) by Lemma
	\ref{diff of two nonhomo eqn}. From Theorem \ref{structure of the general solution}, it follows that there exists $\bq\in \H^n$ such that $  \bphi^{NH}(t)-\bphi_0^{NH}(t)=M(t)\bq$. Then we have equality (\ref{structure of nonhomo eqn}).
\end{proof}

Theorem \ref{Thm of structure of nonhomo eqn}  indicates that if we know a  fundamental solution matrix of (\ref{homegeneous eqn}) and a particular solution  of
(\ref{nonhomegeneous eqn}), then we actually know all possible solutions of (\ref{nonhomegeneous eqn}). The following result further points out that if a fundamental solution matrix of (\ref{homegeneous eqn}) $M(t)$ is known, then the general solution of (\ref{nonhomegeneous eqn}) can be specifically described by method of variation of constants.

\begin{theorem}\label{variation of constants}
	Let $A\in \mathcal{R}(\T,\H^{n\times n})$  and $M(t)$  be a fundamental solution matrix of  (\ref{homegeneous eqn}). Suppose that $\bfunc$ is rd-continuous. Then the general solution of
	(\ref{nonhomegeneous eqn}) is given by
	\begin{equation}\label{general solution of nonhomo eqn}
	\bphi^{NH}(t)=M(t)\bq+M(t)\int_{t_0}^tM^{-1}(\sigma(\tau))\bfunc(\tau)\Delta \tau.
	\end{equation}
	where $\bq=(q_1,q_2,\cdot\cdot\cdot,q_n)^{\top}$ is a constant quaternion vector.
\end{theorem}
\begin{proof}
	Let us look for a solution of (\ref{nonhomegeneous eqn}) in a form similar to
	(\ref{general sol}). Suppose that
	\begin{equation}\label{variation of cont}
	\bphi^{NH}(t)=M(t)\bq(t)
	\end{equation}
	By differentiating  both sides of   (\ref{variation of cont}),  we obtain
	\begin{equation*}
		A(t)M(t)\bq(t)+ \bfunc(t)=M^{\Delta}(t)\bq(t)+M(\sigma(t))\bq^{\Delta}(t).
	\end{equation*}
	Observe that $M^{\Delta}(t)=A(t)M(t)$. Thus
	\begin{equation*}
		\bq^{\Delta}(t)=M^{-1}(\sigma(t))\bfunc(t).
	\end{equation*}
	Therefore
	\begin{equation*}
		\bq(t)=\int_{t_0}^tM^{-1}(\sigma(\tau))\bfunc(\tau)\Delta \tau+\bq,
	\end{equation*}
	where $\bq$ is a constant quaternion vector. Then we obtain an  expression of $\bphi^{NH}(t)$ as follows:
	\begin{equation*}
		\bphi^{NH}(t)=M(t)\bq+M(t)\int_{t_0}^tM^{-1}(\sigma(\tau))\bfunc(\tau)\Delta \tau.
	\end{equation*}
	Now it remains to show that  (\ref{general solution of nonhomo eqn}) is  exactly a solution of (\ref{nonhomegeneous eqn}). We use product rule to differentiate $\bphi^{NH}(t)$:
	\begin{equation*}
		\begin{split}
			\left(\bphi^{NH}\right)^{\Delta}(t) & =M^{\Delta}(t)\bq+M^{\Delta}(t)\int_{t_0}^tM^{-1}(\sigma(\tau))\bfunc(\tau)\Delta \tau+M(\sigma(t)) M^{-1}(\sigma(t))\bfunc(t)\\
			& =A(t)\left( M(t)\bq+M(t)\int_{t_0}^tM^{-1}(\sigma(\tau))\bfunc(\tau)\Delta \tau \right)+\bfunc(t)\\
			&=A(t)\bphi^{NH}(t)+\bfunc(t).
		\end{split}
	\end{equation*}
	The proof is complete.
\end{proof}
\begin{corollary}\label{varia of const with state-transition matrix}
	The solution of initial value problem (\ref{nonhomegeneous eqn}) and (\ref{IVP}) is given by
	\begin{equation}\label{variation of cont transition matrix}
	\bphi^{NH}(t)=\Psi_A(t,t_0) \eta + \int_{t_0}^t \Psi_A(t,\sigma(\tau))\bfunc(\tau)\Delta \tau.
	\end{equation}
	In particular, let $\alpha\in \H $ be a quaternion constant, then the solution of
	$y^{\Delta}(t)=\alpha  y(t)+  f (t) ,~~  y(t_0)=0$
	is  \begin{equation*}
		\begin{split}
			y(t)= & \int_{t_0}^t E_{\alpha}(t,\sigma(\tau))f(\tau)\Delta\tau \\
			=&  E_{\alpha}(t,0)\int_{t_0}^t E_{\alpha}(0,\sigma(\tau))f(\tau)\Delta\tau\\
			=& E_{\alpha}(t,0)\int_{t_0}^t E^{-1}_{\alpha}( \sigma(\tau),0)f(\tau)\Delta\tau.
		\end{split}
	\end{equation*}
\end{corollary}

\section{ Linear QDETS with constant coefficients}\label{S5}

Let $A\in \mathcal{R}(\T,\H^{n\times n})$ be a constant quaternion matrix and suppose that $\bfunc\in  \H^{n\times 1}\otimes \T$ is rd-continuous. In this section, we consider the following  quaternion-valued linear  equations
\begin{equation}\label{nonhomegeneous eqn with const coef}
\bphi^{\Delta}(t){=}A \bphi(t){+}\bfunc(t)
\end{equation}
and its corresponding homogeneous equations
\begin{equation}\label{homegeneous eqn with const coef}
\bphi^{\Delta}(t){=}A \bphi(t).
\end{equation}

\begin{theorem}\label{single solution}
	If $\lambda $ is a right eigenvalue of $A$ and $\bmeta$  is an eigenvector  corresponding to
	$\lambda $. Then $\bphi(t)=\bmeta E_{\lambda}(t,0)$ is a solution of
	(\ref{homegeneous eqn with const coef}).
\end{theorem}
\begin{proof}
	Suppose that $A$ is  regressive, then $\lambda\in  \rc(\T,\H)$ by Theorem
	\ref{equivalent conditions of  regressivity 3}.
	Thus $  E_{\lambda}(t,0)$ is well-defined and therefore
	\begin{equation*}
		\bphi^{\Delta}(t)=    \bmeta \lambda E_{\lambda}(t,0) =A\bmeta  E_{\lambda}(t,0) =A\bphi(t)
	\end{equation*}
	for $t\in\T^{\kappa}$. The proof is complete.
\end{proof}
\begin{theorem}\label{n independent solution}
	If $A$ has $n$ right linearly independent eigenvectors $\bmeta_1,\bmeta_2,\cdot\cdot\cdot,\bmeta_n$ corresponding to right eigenvalues $\lambda_1,\lambda_2, \cdot\cdot\cdot,\lambda_n$ (no matter whether they are similar). Then
	\begin{equation*}\label{fundamental solution matrix by eigenvalue}
		M(t)=\left(\bmeta_1E_{\lambda_1}(t,0),\bmeta_2E_{\lambda_2}(t,0),\cdot\cdot\cdot,\bmeta_nE_{\lambda_n}(t,0)  \right)
	\end{equation*}
	is a fundamental solution matrix of (\ref{homegeneous eqn with const coef}). In particular, if $A$ has $n$ distinct standard eigenvalues, then $\lambda_1,\lambda_2, \cdot\cdot\cdot,\lambda_n$ can be chosen to be the standard eigenvalues of $A$.
\end{theorem}
\begin{proof}
	By Theorem \ref{single solution}, we see that $\bmeta_1E_{\lambda_1}(t,0),\bmeta_2E_{\lambda_2}(t,0),\cdot\cdot\cdot,\bmeta_nE_{\lambda_n}(t,0)$
	are solutions of (\ref{homegeneous eqn with const coef}). It remains to show that they are right linearly independent. Let $W(t)$ be the Wronskian of $M(t)$. Since  $\bmeta_1,\bmeta_2,\cdot\cdot\cdot,\bmeta_n$ are right linearly independent, then
	\begin{equation*}
		W(0)={\qd}M(0)={\qd} (\bmeta_1,\bmeta_2,\cdot\cdot\cdot,\bmeta_n)\neq 0.
	\end{equation*}
	Thus $M(t)$  is a fundamental solution matrix of (\ref{homegeneous eqn with const coef}) by Theorem \ref{whether fundamental solution matrix or not}. If $A$ has $n$ distinct standard eigenvalues  $\lambda_1,\lambda_2, \cdot\cdot\cdot,\lambda_n$, then they are  pairwise non-similar. By   Theorem \ref{thm of q matrix}, we know that their corresponding eigenvectors are right linearly independent.  This completes the proof.
\end{proof}

\begin{example}
	Find a  fundamental solution matrix  of
	\begin{equation}\label{example1 of find funda sol mat}
	\bphi^{\Delta}(t)= A \bphi(t)=   \begin{pmatrix}
	\qi&1\\
	0&1+\qi
	\end{pmatrix}\bphi(t).
	\end{equation}
	for the special time scales  $\T=\mathbb{Z}$.

	By  Theorem \ref{thm of q matrix}, we know that $A$ has two distinct   standard eigenvalues $\lambda_1=\qi,\lambda_2=1+\qi$. Their corresponding eigenvectors are $\bmeta_1=(1,0)^{\top}, \bmeta_2=(1,1)^{\top}$ respectively. By Example \ref{exponential in hZ}, we have
	$E_{\lambda_1}(t,0)=(1+\qi)^t$  and $E_{\lambda_2}(t,0)=(2+\qi)^t$.  Therefore
	\begin{equation*}
		\begin{pmatrix}
			(1+\qi)^t&(2+\qi)^t\\
			0&(2+\qi)^t
		\end{pmatrix}
	\end{equation*}
	is a  fundamental solution matrix of (\ref{example1 of find funda sol mat}).
\end{example}

\begin{example}
	Find a  fundamental solution matrix  of
	\begin{equation}\label{example2 of find funda sol mat}
	\bphi^{\Delta}(t)= A \bphi(t)=   \begin{pmatrix}
	\qi&\qj\\
	0& \qi
	\end{pmatrix}\bphi(t).
	\end{equation}
	for the special time scales  $\T=\mathbb{Z}$.

	By  Theorem \ref{thm of q matrix}, we know that $A$    has a repeated standard eigenvalue  $\lambda =\qi$. Although $A$ only has one standard eigenvalue,  there are two right linearly independent eigenvectors  $\bmeta_1=(1,0)^{\top}, \bmeta_2=( \frac{\qk}{2},1)^{\top}$ corresponding to  $\lambda =\qi$. Therefore
	\begin{equation*}
		\begin{pmatrix}
			(1+\qi)^t&\frac{\qk (1+\qi)^t}{2}\\
			0&(1+\qi)^t
		\end{pmatrix}
	\end{equation*}
	is a  fundamental solution matrix of (\ref{example2 of find funda sol mat}).
\end{example}

From Example 7.4 in \cite{zhang1997quaternions}, we know that not every $n\times n$ constant quaternion matrix has $n$ right linearly independent eigenvectors. In this case, Theorem
\ref{n independent solution} is of no use any more. The Putzer's  algorithm for the classical case in \cite{bohner2001dynamic} is  generalized to  quaternion dynamic equations on time scales. Since Putzer's  algorithm  avoids the computing of  eigenvectors. So  it is particularly useful for quaternion matrices that do not have  $n$ right linearly independent eigenvectors.

\begin{theorem}\label{Thm of putzer algorithm}
	Let $A\in \mathcal{R}(\T,\H^{n\times n})$ be a constant quaternion matrix.   If there exists
	quaternion constants $\alpha_1,\alpha_2,\cdot\cdot\cdot,\alpha_m \in \rc(\T,\H)$  such that $P_m=0$,  where $P_0,P_1,P_2,\cdot\cdot\cdot,P_m$ are recursively defined by $P_0=I_n$ and
	\begin{equation*}
		\begin{split}
			P_{k}= &  AP_{k-1}-P_{k-1}\alpha_{k} \\
			=&  A^k + \sum_{j=1}^k(-1)^jA^{k-j}\sum_{1\leq  i_1< i_2\cdot\cdot\cdot<i_j\leq k} \alpha_{i_1}\alpha_{i_2}\cdot\cdot\cdot\alpha_{i_j}
		\end{split}
	\end{equation*}
	{for} $1\leq k\leq m $.
	Then
	\begin{equation}\label{putzer algorithm}
	\Psi_A(t,t_0)=\sum_{k=0}^{m-1}P_k\varphi_{k+1}(t),
	\end{equation}
	where $\bvphi(t):=(\varphi_1(t),\varphi_1(t),\cdot\cdot\cdot,\varphi_m(t))^{\top}$ is the solution of the following initial value problem
	\begin{equation}\label{vphi def by IVP}
	\bvphi^{\Delta}(t)=  \begin{pmatrix}
	\alpha_1&0&0&\dots&0\\
	1 &\alpha_2 & 0&\ddots&\vdots\\
	0&1&  \alpha_3&\ddots&\vdots\\
	\vdots&   \ddots&\ddots&\ddots&0\\
	0&\dots&0&1&\alpha_m
	\end{pmatrix}
	\bvphi(t),~~~\bvphi(t_0)=
	\begin{pmatrix}
	1\\
	0\\
	0\\
	\vdots\\
	0
	\end{pmatrix}.
	\end{equation}
\end{theorem}

\begin{proof}
	From statement 6 of  Theorem \ref{thm of q matrix} and  Lemma
	\ref{equivalent conditions of  regressivity 2}  we see that the coefficient matrix of (\ref{vphi def by IVP}) is regressive and therefore (\ref{vphi def by IVP}) has a unique solution $\bvphi(t):=(\varphi_1(t),\varphi_1(t),\cdot\cdot\cdot,\varphi_m(t))^{\top}$. Let
	$\bphi(t)$ be the right-hand side of (\ref{putzer algorithm}). Then
	\begin{equation*}
		\begin{split}
			\bphi^{\Delta}(t)-A \bphi(t)&  = \sum_{k=0}^{m-1}P_k\varphi^{\Delta}_{k+1}(t)-A\sum_{k=0}^{m-1}P_k\varphi_{k+1}(t) \\
			& =P_0\alpha_1 \varphi_1(t)+\sum_{k=1}^{m-1}P_k\left(\varphi_k (t) +\alpha_{k+1}\varphi _{k+1}(t)\right)-A\sum_{k=0}^{m-1}P_k\varphi_{k+1}(t)\\
			&=\sum_{k=1}^{m-1}P_k \varphi_k (t)-\sum_{k=0}^{m-1}\left( AP_k-P_k \alpha_{k+1}  \right) \varphi _{k+1}(t)\\
			&=\sum_{k=1}^{m-1}P_k \varphi_k (t)-\sum_{k=0}^{m-1}P_{k+1} \varphi _{k+1}(t)\\
			&=-P_m\varphi _{m}(t)=0.
		\end{split}
	\end{equation*}
	The last  equality is a consequence of $ P_m=0$. Since $\bvphi(t_0)=\varphi_1(t_0)P_0=I$. Thus
	$\bvphi(t)$ is a fundamental solution matrix of  (\ref{homegeneous eqn}).
	Therefore
	\begin{equation*}
		\Psi_A(t,t_0)= \bvphi(t) \bvphi(t_0)=\sum_{k=0}^{m-1}P_k\varphi_{k+1}(t),
	\end{equation*}
	which completes the proof.
\end{proof}

\begin{example}\label{example1 of Putzer algorithm}
	Find the state-transition matrix of the quaternion dynamic equations
	\begin{equation*}
		\bphi^{\Delta}(t)=A\bphi(t)=
		\begin{pmatrix}
			\qi&1&0\\
			0&\qj&0\\
			0&1&\qk
		\end{pmatrix}\bphi(t)
	\end{equation*}
	for the special time scales $\T=\mathbb{R}$  and $\T=\mathbb{Z}$.
	
	Since $\chi_A$ is not diagonalizable, then $A$ does not have $n$ right linearly independent eigenvectors belonging to right eigenvalues. So Theorem \ref{n independent solution}  does not apply to this problem.
	
	Let $\alpha_1=\qi, \alpha_2=-\qi, \alpha_3=\qj$ and $P_0=I_3$. Then we can easily get that
	\begin{equation*}
		P_1= \begin{pmatrix}
			0&1&0\\
			0&\qj-\qi&0\\
			0&1&\qk-\qi
		\end{pmatrix}, ~~ ~~
		P_2= \begin{pmatrix}
			0&\qj+\qi&0\\
			0&0&0\\
			0&\qj+\qk&0
		\end{pmatrix}
	\end{equation*}
	and $P_3=0$.  Now we need to  solve IVP
	\begin{equation*}
		\bvphi^{\Delta}(t)=  \begin{pmatrix}
			\qi&0&0\\
			1&-\qi&0\\
			0&1&\qj
		\end{pmatrix}\bvphi (t),~~~\bvphi(t_0)=
		\begin{pmatrix}
			1\\
			0\\
			0
		\end{pmatrix}.
	\end{equation*}
	
	If  $\T=\mathbb{R}$ and $t_0=0$, then $\varphi_1(t)=e^{\qi  t}$ and
	\begin{equation*}
		\varphi_2^{\Delta}(t)=-\qi  \varphi_2(t)+e^{\qi  t},~~~  \varphi_2(0)=0.
	\end{equation*}
	Then by Corollary \ref{varia of const with state-transition matrix}, we have $\varphi_2(t)=\sin t$ and therefore
	\begin{equation*}
		\varphi_3^{\Delta}(t)=\qj \varphi_3(t)+\sin t,~~~  \varphi_3(0)=0.
	\end{equation*}
	Thus
	\begin{equation*}
		\varphi_3(t)= \frac{1}{4} e^{-\qj t} \left(-2 \qj e^{2 \qj t} t+e^{2 \qj t}-1\right).
	\end{equation*}
	Then we obtain
	\begin{equation}\label{state-transition matrix for R of example 1}
	\begin{split}
	\Psi_A(t, 0) = \sum_{k=0}^{2}P_k\varphi_{k+1}(t)
	= \begin{pmatrix}
	e^{\qi t}& \frac{t}{2} \left(e^{\qi t}-\qk e^{-\qi t} \right)+\frac{1+\qk}{2}\sin t&0\\
	0&e^{\qj t}&0\\
	0& \frac{t}{2} \left(e^{\qj t}+\qi e^{\qj t} \right)+\frac{1-\qi}{2}\sin t&e^{\qk t}
	\end{pmatrix}.
	\end{split}
	\end{equation}
	By direct computation, we see that both $ \Psi_A^{\Delta}(t, 0)$ and $A\Psi_A(t, 0)$ are equal to
	\begin{equation*}
		\begin{pmatrix}
			\qi  e^{\qi t}& \frac{t}{2} \left(\qi e^{\qi t}+\qj e^{-\qi t} \right)+\frac{1}{2}\left(  e^{\qi t}+ e^{\qj t} \right)&0\\
			0&\qj e^{\qj t}&0\\
			0& \frac{t}{2} \left(\qj e^{\qj t}+\qk e^{\qj t} \right)+\frac{1}{2}\left( e^{\qj t}+  e^{\qk t}\right)&\qk e^{\qk t}
		\end{pmatrix}.
	\end{equation*}
	That means that (\ref{state-transition matrix for R of example 1}) is exactly the state-transition matrix.
	
	If  $\T=\mathbb{Z}$ and $t_0=0$, then $\varphi_1(t)= (1+\qi)^t$ and
	\begin{equation*}
		\varphi_2^{\Delta}(t)=-\qi  \varphi_2(t)+(1+\qi)^t,~~~  \varphi_2(0)=0.
	\end{equation*}
	Then we have $\varphi_2(t)=\frac{\qi}{2}  \left( (1-\qi)^t-(1+\qi)^t\right)$ and therefore
	\begin{equation*}
		\varphi_3^{\Delta}(t)=\qj \varphi_3(t)+\frac{1}{2} \qi \left( (1-\qi)^t-(1+\qi)^t\right),~~~  \varphi_3(0)=0.
	\end{equation*}
	Thus
	\begin{equation*}
		\begin{split}
			\varphi_3(t)= &  \frac{1}{4}\left(1-2^{-t}(1-\qi-\qj-\qk)^t\right)(1+\qi-\qj+\qk) \\
			&  +\frac{1}{4}\left(1-2^{-t}(1+\qi-\qj+\qk)^t\right)(-1+\qi+\qj+\qk) .
		\end{split}
	\end{equation*}
	Then we obtain
	\begin{equation}\label{state-transition matrix for Z of example 1}
	\begin{split}
	\Psi_A(t, 0) =  \begin{pmatrix}
	(1+\qi)^t &   \frac{\qi}{2}\gamma_1+2^{-t}(1+\qi)\gamma_2-\frac{1-\qi+\qj+\qk}{2}&0\\
	0&(1+\qi)^t+\frac{1-\qk}{2}\gamma_1 &0\\
	0&\frac{\qi}{2}\gamma_1+2^{-t}(\qi+\qk)\gamma_2 -\frac{1-\qi-\qj+\qk}{2}&(1+\qi)^t+\frac{1+\qj}{2}\gamma_1
	\end{pmatrix},
	\end{split}
	\end{equation}
	where $\gamma_1=    (1-\qi)^t-(1+\qi)^t $, $\gamma_2=  (1+\qi-\qj+\qk)^{t-1}-(1-\qi-\qj-\qk)^{t-1} $.
	By direct computation, we see that both $ \Psi_A^{\Delta}(t, 0)$ and $A\Psi_A(t, 0)$ are equal to
	\begin{equation*}
		\begin{split}
			\begin{pmatrix}
				\qi (1+\qi)^t & (1+\qi)^t  - \frac{\qk}{2}\gamma_1+2^{-t}( \qi-1)\gamma_2-\frac{1+\qi-\qj+\qk}{2}&0\\
				0&\qj (1+\qi)^t+\frac{\qj-\qi}{2}\gamma_1 &0\\
				0& (1+\qi)^t +\frac{1+\qj-\qk}{2}\gamma_1+2^{-t}(\qj-1)\gamma_2 +\frac{1-\qi+\qj-\qk}{2}&\qk (1+\qi)^t+\frac{\qk-\qi}{2}\gamma_1
			\end{pmatrix}.
		\end{split}
	\end{equation*}
	That means that (\ref{state-transition matrix for Z of example 1}) is exactly the state-transition matrix.
\end{example}

\begin{example}\label{example2 of Putzer algorithm}
	Find the state-transition matrix of the quaternion dynamic equations
	\begin{equation*}
		\bphi^{\Delta}(t)=A\bphi(t)=
		\begin{pmatrix}
			\qi&\qj&\qj\\
			\qk&1&\qk\\
			0&0&1
		\end{pmatrix}\bphi(t)
	\end{equation*}
	for the special time scales $\T=\mathbb{R}$  and $\T=\mathbb{Z}$.

	Let $\alpha_1=1, \alpha_2=0, \alpha_3=1+\qi, \alpha_4=1-\qi$ and $P_0=I_3$. Then we can easily get that
	\begin{equation*}
		P_1= \begin{pmatrix}
			\qi-1&\qj&\qj\\
			\qk&0&\qk\\
			0&0&0
		\end{pmatrix},~~
		P_2= \begin{pmatrix}
			-1&\qk&\qi+\qk\\
			\qj&-\qi&\qk-\qi\\
			0&0&0
		\end{pmatrix},~~
		P_3= \begin{pmatrix}
			0&-2\qj&-2\qj\\
			0&-2&-2\\
			0&0&0
		\end{pmatrix}
	\end{equation*}
	and $P_4=0$.
	Now we need to  solve IVP
	\begin{equation*}
		\bvphi^{\Delta}(t)=  \begin{pmatrix}
			1&0&0&0\\
			1&0&0&0\\
			0&1&1+\qi&0\\
			0&0&1&1-\qi
		\end{pmatrix}\bvphi (t),~~~\bvphi(t_0)=
		\begin{pmatrix}
			1\\
			0\\
			0\\
			0
		\end{pmatrix}.
	\end{equation*}
	
	If  $\T=\mathbb{R}$ and $t_0=0$, then $\varphi_1(t)=e^{t}, \varphi_2(t)=e^{t}-1$ and
	\begin{equation*}
		\varphi_3^{\Delta}(t)=(1+\qi)  \varphi_3(t)+ e^t-1,~~~  \varphi_3(0)=0.
	\end{equation*}
	Then  we have $\varphi_3(t)=\left(-\frac{1}{2}-\frac{\qi}{2}\right) \left(-(1+\qi) e^t+e^{(1+\qi) t}+\qi\right)$ and therefore
	\begin{equation*}
		\varphi_4^{\Delta}(t)=(1-\qi) \varphi_4(t)+ \left(-\frac{1}{2}-\frac{\qi}{2}\right) \left(-(1+\qi) e^t+e^{(1+\qi) t}+\qi\right),~~  \varphi_4(0)=0.
	\end{equation*}
	Thus
	\begin{equation*}
		\varphi_4(t)=- \frac{1+\qi}{4}  e^{(1-\qi) t}+e^t- \frac{1-\qi}{4}   e^{(1+\qi) t}-\frac{1}{2}.
	\end{equation*}
	Using the  Putzer's algorithm  (\ref{putzer algorithm}),  we can obtain
	\begin{equation}\label{state-transition matrix for R of example 2}
	\begin{split}
	\Psi_A(t, 0)
	= \begin{pmatrix}
	\frac{1-\qi}{2}+ \frac{1+\qi}{2}\gamma_1& \frac{\qk-\qj}{2}+\gamma_2&\qj \gamma_3+\gamma_4-e^t\\
	\frac{\qj-\qk}{2}+\frac{\qk-\qj}{2}\gamma_1 &\frac{1-\qi}{2}-\qj\gamma_2&\qi \gamma_3-\qj\gamma_4-(1-\qj-\qk)e^t\\
	0&0&e^{ t}
	\end{pmatrix},
	\end{split}
	\end{equation}
	where $\gamma_1=e^{(1+\qi)t}$, $\gamma_2= \frac{\qj-\qk}{2}e^{(1-\qi)t}$, $\gamma_3=\frac{\qk-1-\qi-\qj}{2} $,  $\gamma_4=e^{(1+\qi)t}\frac{1-\qi+\qj-\qk}{2}$.  The result   is consistent with the result of Example 6.3 in \cite{kou2015linear2}.

	If  $\T=\mathbb{Z}$ and $t_0=0$, then $\varphi_1(t)= 2^t,  \varphi_2(t)=2^{t}-1$ and
	\begin{equation*}
		\varphi_3^{\Delta}(t)=(1+\qi)  \varphi_3(t)+ 2^t-1,~~~  \varphi_3(0)=0.
	\end{equation*}
	Then we have
	\begin{equation*}
		\begin{split}
			\varphi_3(t)= &  (2+\qi)^t  \int_{0}^t (2+\qi)^{-(\tau+1)}(2^{\tau}-1)\Delta \tau\\
			=&   (2+\qi)^t\sum _{\tau=0}^{t-1}(2+\qi)^{-(\tau+1)}(2^{\tau}-1)\\
			=&- \frac{1+\qi}{2}\left(\qi -(1+\qi) 2^t+(2+\qi)^t \right),
		\end{split}
	\end{equation*}
	and therefore
	\begin{equation*}
		\varphi_4^{\Delta}(t)=(1-\qi) \varphi_4(t)- \frac{1+\qi}{2}\left(\qi -(1+\qi) 2^t+(2+\qi)^t \right),~~~  \varphi_4(0)=0.
	\end{equation*}
	Thus
	\begin{equation*}
		\begin{split}
			\varphi_4(t)= &   (2-\qi)^t  \int_{0}^t  (2-\qi)^{-\tau-1}  \varphi_3(\tau) \Delta \tau \\
			=  &   \frac{1}{4}\left(2^{2+t}-(1+\qi)(2-\qi)^t+(\qi-1)(2+\qi)^t-2 \right) .
		\end{split}
	\end{equation*}
	Then we obtain
	\begin{equation}\label{state-transition matrix for Z of example 2}
	\begin{split}
	\Psi_A(t, 0)
	= \begin{pmatrix}
	\frac{1-\qi}{2}+ \frac{1+\qi}{2}\gamma_1& \frac{\qk-\qj}{2}+\gamma_2&\qj \gamma_3+\gamma_4-2^t\\
	\frac{\qj-\qk}{2}+\frac{\qk-\qj}{2}\gamma_1 &\frac{1-\qi}{2}-\qj\gamma_2&\qi \gamma_3-\qj\gamma_4-(1-\qj-\qk)2^t\\
	0&0&2^{ t}
	\end{pmatrix},
	\end{split}
	\end{equation}
	where $\gamma_1= (2+\qi)^t $, $\gamma_2= \frac{\qj-\qk}{2} {(2-\qi)^t}$, $\gamma_3=\frac{\qk-1-\qi-\qj}{2} $,  $\gamma_4= {(2+\qi)^t}\frac{1-\qi+\qj-\qk}{2}$.
	By direct computation, we see that both $ \Psi_A^{\Delta}(t, 0)$ and $A\Psi_A(t, 0)$ are equal to
	\begin{equation*}
		\begin{split}
			\begin{pmatrix}
				\qi (2+\qi)^t &  \qj (2-\qi)^t&(1+\qi)\gamma_4-2^t\\
				\qk (2+\qi)^t & (2-\qi)^t &(\qk-\qj)\gamma_4-(1-\qj-\qk)2^t\\
				0& 0&2^t
			\end{pmatrix}.
		\end{split}
	\end{equation*}
	That means that (\ref{state-transition matrix for Z of example 2}) is exactly the state-transition matrix.
	
\end{example}

Theorem \ref{Thm of putzer algorithm}  is a generalization of Theorem 5.35 in \cite{bohner2001dynamic}. In the classical case, by Cayley-Hamilton theorem,
\begin{equation*}
	\prod_{j=1}^n(A-\lambda_jI)=0.
\end{equation*}
where  $A\in \C^{n\times n}$ and $\lambda_j, (1 \leq j \leq  n)$  are eigenvalues of $A$. So $\alpha_1,\alpha_2,\cdot\cdot\cdot,\alpha_m$ can be chosen as the eigenvalues of $A\in \C^{n\times n}$.   In the quaternion case, however, the selection of $\alpha_k,(1\leq k\leq m)$ is more difficult. What should be clear to us is that the less $m$  the less calculation.

We say that  $h(z)=z^m+z^{m-1}\beta_1+\cdot\cdot\cdot+\beta_m$ is an  annihilating polynomial of
quaternion matrix $A\in\H^{n\times n}$ if
\begin{equation*}
	h(A)=A^m+A^{m-1}\beta_1+\cdot\cdot\cdot+I_n\beta_m=0,
\end{equation*}
where $\beta_1,\beta_2, \cdot\cdot\cdot,\beta_m\in \H$. To authors' best knowledge, there are
(at least) two annihilating polynomials for every $A\in\H^{n\times n}$. The first one which was
presented by Zhang \cite{zhang1997quaternions} is $\mathrm{ch}_{\chi_A}(z)$, the characteristic
polynomial of $\chi_A$. In this case, $m=2n$ and $\alpha_1,\alpha_2,\cdot\cdot\cdot,\alpha_m$ are exactly the standard eigenvalues of $A$. The other one is called minimal  polynomial which was given by Rodman \cite{rodman2014topics}. The coefficients of minimal  polynomial in \cite{rodman2014topics} are confined to be real. Thus, there   may  be   some other annihilating   polynomials (with quaternion  coefficients), which possess less degree than minimal  polynomial. Since $m$-degree  minimal  polynomial possesses real coefficients, then it has $m$  complex roots and the Vieta's formula holds. Therefore, $\alpha_1,\alpha_2,\cdot\cdot\cdot,\alpha_m$  can be chosen as the complex roots of  minimal  polynomial. Unfortunately, there is no explicit expression for  minimal  polynomial of  quaternion matrices  until now. On the other hand,  even when  we    know   an annihilating polynomial  $h(z)=z^m+z^{m-1}\beta_1+\cdot\cdot\cdot+\beta_m$  (with  quaternion  coefficients) of $A$, we still can not find $\alpha_1,\alpha_2,\cdot\cdot\cdot,\alpha_m$. In fact, we need to find $\alpha_1,\alpha_2,\cdot\cdot\cdot,\alpha_m$ such that
\begin{equation}\label{alpha and beta}
(-1)^k \sum_{1\leq  i_1< i_2\cdot\cdot\cdot<i_k\leq m} \alpha_{i_1}\alpha_{i_2}\cdot\cdot\cdot\alpha_{i_k}=\beta_k
\end{equation}
for $1\leq k  \leq m$. Thus, $\alpha_1,\alpha_2,\cdot\cdot\cdot,\alpha_m$ may not exist. Even if they exist, we can not conclude that they are roots of $h(z)$
(see Example \ref{Vieta formula not hold}). Even if they are  roots of $h(z)$, we still can not find them. This is because that  the number of  zeros of   quaternion  polynomials  is indeterminate and the computing of  zeros of   quaternion  polynomials is complicated. For details of zeros of  quaternion  polynomials, please refer to \cite{serodio2001zeros,serodio2001computing,pogorui2004structure,janovska2010note}.

\begin{example}\label{Vieta formula not hold}
	Let $h(z)=z^2+z\beta_1+\beta_2$ and $\alpha_1=\qi,\alpha_2=-\qj$. Then
	$\alpha_1,\alpha_2,\beta_1,\beta_2$ satisfy (\ref{alpha and beta}), but
	\begin{equation*}
		h(\alpha_1)=\qi^2+\qi (\qj-\qi)+\qk=2\qk\neq 0.
	\end{equation*}
\end{example}

In practice, $\alpha_1,\alpha_2,\cdot\cdot\cdot,\alpha_m$  are usually chosen to be the eigenvalues of $A\in\H^{n\times n}$ (see Example 
\ref{example1 of Putzer algorithm} and \ref{example2 of Putzer algorithm}). The value of $m$ does not need to be as large as $2n$. By iterative computing, we can always  get more and more  succinct  $P_k$  as $k$  increases.
Although there are  some  theoretical challenges on selections of  $\alpha_1,\alpha_2,\cdot\cdot\cdot,\alpha_m$, the Putzer's algorithm is still feasible. The theoretical challenges give  us something to focus on and work toward.

The method of variation of constants for one-dimensional case has been used many times in  Example  \ref{example1 of Putzer algorithm} and \ref{example2 of Putzer algorithm}. We now present an example to illustrate the feasibility of the  method of variation of constants  for higher dimensional case.

\begin{example}
	Find the solution of the initial value problem
	\begin{equation}\label{variation of constants  for higher dimension}
	\bphi^{\Delta}(t)=A\bphi(t)  +\bfunc(t)=\begin{pmatrix}
	\qj&0\\
	0&\qk
	\end{pmatrix}\bphi(t) +\begin{pmatrix}
	\qi\\
	t\qj
	\end{pmatrix},~~\bphi(0)= \begin{pmatrix}
	\qj\\
	\qk
	\end{pmatrix}.
	\end{equation}
	for the special time scale $\T=\mathbb{Z}$.
	
	It is easy to get that
	\begin{equation*}
		\Psi_A(t, 0)=\begin{pmatrix}
			(1+\qj)^t&0\\
			0&(1+\qk)^t
		\end{pmatrix}.
	\end{equation*}
	By Corollary \ref{varia of const with state-transition matrix},
	\begin{equation*}
		\begin{split}
			\bphi (t)  =&  \Psi_A(t, 0) \bphi(0)+  \Psi_A(t, 0) \int_0^t \Psi^{-1}_A(\sigma(\tau), 0)\bfunc(\tau)\Delta\tau\\
			=&    \Psi_A(t, 0) \bphi(0)+\Psi_A(t, 0)\int_0^t\begin{pmatrix}
				(1+\qj)^{-\tau-1}&0\\
				0&(1+\qk)^{-\tau-1}
			\end{pmatrix}\begin{pmatrix}
				\qi\\
				t \qj
			\end{pmatrix}\Delta\tau\\
			=&  \Psi_A(t, 0) \bphi(0)+\Psi_A(t, 0)\begin{pmatrix}
				\sum _{\tau=0}^{t-1}(1+\qj)^{-\tau-1}\qi \\
				\sum _{\tau=0}^{t-1}\tau(1+\qk)^{-\tau-1}\qj
			\end{pmatrix}\\
			=&   \Psi_A(t, 0) \bphi(0)+\begin{pmatrix}
				(1+\qj)^t&0\\
				0&(1+\qk)^t
			\end{pmatrix}\begin{pmatrix}
				\left(1-(1+\qj)^{-t}\right)\qk \\
				(1+\qk)^{-t}\left(1-(1+\qk)^t+t\qk\right)\qj
			\end{pmatrix}\\
			=&  \begin{pmatrix}
				\qj (1+\qj)^t\\
				\qk (1+\qk)^t
			\end{pmatrix}+\begin{pmatrix}
				\left( (1+\qj)^{t}-1\right)\qk \\
				\left(1-(1+\qk)^t+t\qk\right)\qj
			\end{pmatrix}\\
			=&\begin{pmatrix}
				\qj (1+\qj)^t+  \left( (1+\qj)^{t}-1\right)\qk \\
				\qk (1+\qk)^t+  \left(1-(1+\qk)^t+t\qk\right)\qj
			\end{pmatrix}.\\
		\end{split}
	\end{equation*}
	
	Note that
	\begin{equation*}
		\begin{split}
			\phi_1^{\Delta}(t)= & \phi_1(t+1)- \phi_1(t) \\
			=  &   \qj (1+\qj)^{t+1}+  \left( (1+\qj)^{t+1}-1\right)\qk-\qj (1+\qj)^t+  \left( (1+\qj)^{t}-1\right)\qk\\
			= &(1+\qj)^{t}(1-\qi)\\
			= &\qj \phi_1(t)+\qi
		\end{split}
	\end{equation*}
	and
	\begin{equation*}
		\begin{split}
			\phi_2^{\Delta}(t)= & \phi_2(t+1)- \phi_2(t) \\
			=  &  \qk (1+\qk)^{t+1}+  \left(1-(1+\qk)^{t+1}+(t+1)\qk\right)\qj-  \qk (1+\qk)^t+  \left(1-(1+\qk)^t+t\qk\right)\qj\\
			= &(1+\qk)^{t}(\qi-1)-\qi\\
			= &\qk \phi_2(t)+t\qj.
		\end{split}
	\end{equation*}
	Thus   $ \bphi (t)$ is exactly the solution of
	(\ref{variation of constants  for higher dimension}).
\end{example}

\section{Conclusion}\label{S6}

In this paper, we establish the basic theory of linear quaternion dynamic equations on time scales (QDETS). It not only generalizes the theory of quaternion differential equations (QDEs) but also extends the theory of dynamic equations on time scales (DETS).
Employing the newly defined Wronskian determinant, the Liouville's formula for QDETS is derived, thereby giving the structure of general solutions of QDETS.
We present the Putzer's algorithm to compute fundamental solution matrix for homogeneous QDETS.  The Putzer's algorithm is applicable to all  homogeneous QDETS with constant coefficient matrices. It is particularly useful for quaternion coefficient matrices which are not diagonalizable. Furthermore, the variation of constants formula of solving   the nonhomogeneous QDETS is also derived. Importantly, examples are given in each sections to illustrate our results.


\end{document}